\documentclass[11pt]{amsart}
\usepackage{amssymb,latexsym,epsfig,color}
 
\newtheorem{theorem}{Theorem}[section]

\newtheorem{corollary}[theorem]{Corollary}

\newtheorem{lemma}[theorem]{Lemma}

\theoremstyle{definition}
\newtheorem{definition}[theorem]{Definition}
\newtheorem{remark}[theorem]{Remark}
\newtheorem{example}[theorem]{Example}

\def\NN{\ensuremath{\mathbb{N}}}

\def\RR{\ensuremath{\mathbb{R}}}
 \def\oR{\ensuremath{\mathbb{R}}}

\newcommand{\CC}{{\mathbb C}}

\newcommand{\B}{{\mathcal B}}

\renewcommand{\TH}{{\textup{TH}}}
\newcommand{\OO}{{\mathcal O}}

\newcommand{\PPP}{{\mathcal P}}
\newcommand{\SSS}{{\mathcal S}}
\def\T{\ensuremath{\mathcal{T}}}

\def\H{\ensuremath{\mathcal{H}}}
\def\I{\ensuremath{\mathcal{I}}}
\def\F{\ensuremath{\mathcal{F}}}

\def\C{\ensuremath{\mathcal{C}}}
\def\D{\ensuremath{\mathcal{D}}}

 \def\M{\ensuremath{\mathcal{M}}}

\def\one{\ensuremath{{\bf{1}}}}

\def\uu{\ensuremath{{\bf{u}}}}

\def\ww{\ensuremath{{\bf{w}}}}
\def\xx{\ensuremath{{\bf{x}}}}
\def\yy{\ensuremath{{\bf{y}}}}

\def\GF{\ensuremath{\textup{GF}}}

\def\CUT{\ensuremath{\textup{CUT}}}
 \def\MET{\ensuremath{\textup{MET}}}

\def\CYC{\ensuremath{\textup{CYC}}}

\def\span{\ensuremath{\textup{span}}}

\def\max{\ensuremath{\textup{max}}}
\def\min{\ensuremath{\textup{min}}}

\def\conv{\ensuremath{\textup{conv}}}
\def\deg{\ensuremath{\textup{deg}}}

\newcommand{\ignore}[1]{}

\title[SDP hierarchy for cycles in binary matroids and cuts in
graphs]{A new semidefinite programming hierarchy for cycles in binary  
  matroids\\and cuts in graphs}

\author[J. Gouveia]{Jo{\~a}o Gouveia}
\address{Department of Mathematics, University of Washington, Box
  354350, Seattle, WA 98195, USA, and CMUC, Department of Mathematics,
  University of Coimbra, 3001-454 Coimbra, Portugal}
\email{jgouveia@math.washington.edu} 
\author[M. Laurent]{Monique Laurent}
\address{CWI, Science Park 123, 1098 XG Amsterdam, The Netherlands,
 and Department of Econometrics and Operations Research, Tilburg University, Tilburg, The Netherlands}
\email{monique@cwi.nl}
\author[P.A. Parrilo]{Pablo A.\ Parrilo}
\address{Department of Electrical Engineering and Computer Science,
  Laboratory for Information and Decision Systems, Massachusetts
  Institute of Technology, 77 Massachusetts Avenue, Cambridge, MA
  02139-4307, USA} 
\email{parrilo@mit.edu} 
\author[R. Thomas]{Rekha Thomas}
\address{Department of Mathematics, University of Washington, Box
  354350, Seattle, WA 98195, USA} 
\email{thomas@math.washington.edu}
 \thanks{Gouveia, Parrilo and Thomas were partially supported by 
   the NSF Focused 
   Research Group grants DMS-0757371 and DMS-0757207. 
   Gouveia was also partially
   supported by Funda{\c c}{\~ a}o para a Ci{\^ e}ncia e Tecnologia and
   Thomas by the Robert R. and Elaine K. Phelps Endowment at the
   University of Washington.}  
\date{\today}

\begin{document}

\begin{abstract} The theta bodies of a polynomial ideal are a series
  of semidefinite programming relaxations of the convex hull of the
  real variety of the ideal. In this paper we construct the theta
  bodies of the vanishing ideal of cycles in a binary matroid. Applied
  to cuts in graphs, this yields a new hierarchy of semidefinite
  programming relaxations of the cut polytope of the graph. If the
  binary matroid avoids certain minors we can characterize when the
  first theta body in the hierarchy equals the cycle polytope of the
  matroid. Specialized to cuts in graphs, this result solves a
  problem posed by Lov{\'a}sz.
\end{abstract}

\maketitle

\section{Introduction}

A central question in combinatorial optimization is to understand the
polyhedral structure of the convex hull, $\conv(S)$, of a finite set
$S\subseteq \RR^n$.  A typical instance is when $S$ is the set of
incidence vectors of a finite set of objects over which one is
interested to optimize; think for instance of the problem of finding a
shortest tour, a maximum independent set, or a maximum cut in a graph.
As for hard combinatorial optimization problems one cannot hope in
general to be able to find the complete linear description of the
polytope $\conv(S)$, the objective is then to find good and efficient
approximations of this polytope. Such approximations could be
polyhedra, obtained by considering classes of valid linear
inequalities. In recent years more general convex semidefinite
programming (SDP) relaxations have been considered, which sometimes
yield much tighter approximations than those from LP methods. This was
the case for instance for the approximation of stable sets and
coloring in graphs via the theta number introduced by Lov\'asz
\cite{Lo79}, and for the approximation of the max-cut problem by
Goemans and Williamson \cite{GW95}. See e.g. \cite{LR05} for an
overview.  These results spurred intense research activity on
constructing stronger SDP relaxations for combinatorial optimization
problems (cf. \cite{LS91, SA90,Las01b,Parrilo:spr,Lau03a,LR05}).  In
this paper we revisit the hierarchy of SDP relaxations proposed by
Gouveia et al. \cite{GPT} which was inspired by a question of Lov\'asz
\cite{Lovasz}.  To present it we need some definitions.

Let $I\subseteq \RR[\xx]$ be an ideal and $V_\RR(I)=\{\xx\in\RR^n\mid
f(\xx)=0 \ \forall f\in I\}$ be its real variety.  Throughout
$\RR[\xx]$ denotes the ring of multivariate polynomials in $n$
variables $\xx=(x_1,\ldots,x_n)$ over $\RR$ and $\RR[\xx]_d$ its
subspace of polynomials of degree at most $d\in\NN$.  As the convex
hull of $V_\RR(I)$ is completely described by the (linear) polynomials
$f\in \RR[\xx]_1$ that are non-negative on $V_\RR(I)$, relaxations of
$\conv(V_\RR(I))$ can be obtained by considering sufficient conditions
for the non-negativity of linear polynomials on $V_\RR(I)$.

A polynomial $f\in\RR[\xx]$ is said to be a {\em sum of squares} ({\em
  sos}, for short) if $f=\sum_{i=1}^t g_i^2$ for some polynomials
$g_i\in\RR[\xx]$. Moreover, $f$ is said to be {\em sos modulo the
  ideal $I$} if $f = \sum_{i=1}^t g_i^2 + h$ for some polynomials
$g_i\in\RR[\xx]$ and $h\in I$. In addition, if each $g_i$ has degree
at most $k$, then we say that $f$ is {\em $k$-sos modulo $I$}.
Obviously any polynomial which is $k$-sos modulo $I$ is non-negative
over $V_\RR(I)$.  Following \cite{GPT}, for each $k \in \NN$, define
the set
\begin{equation}
\label{THk}
\TH_k(I):=\{\xx\in\RR^n\mid f(\xx)\ge 0 \ \text{ for all } f \in
\RR[\xx]_1  \ k{\text{-sos modulo }} I\},
\end{equation}
called the {\em $k$-th
  theta body} of the ideal $I$. Note that $\TH_k(I)$ is a (convex)
relaxation of $\conv(V_\RR(I))$, with $$\conv(V_\RR(I))\subseteq
\TH_{k+1}(I)\subseteq \TH_k(I).$$ 
% As testing sums of squares of
% polynomials can be reformulated as a semidefinite program, we obtain a
% hierarchy of (convex) semidefinite relaxations of $\conv(\V_\RR(I))$:
The ideal $I$ is said to be {\em $\TH_k$-exact} if the equality
$\overline{\conv(V_\RR(I))} = \TH_k(I)$ holds.  The theta bodies
$\TH_k(I)$ were introduced in \cite{GPT}, inspired by a question of
Lov\'asz \cite[Problem~8.3]{Lovasz} asking to characterize
$\TH_k$-exact ideals, in particular when $k=1$.

This question of Lov\'asz was motivated by the following result about
stable sets in graphs: The stable set ideal of a graph $G=(V,E)$ is
$\TH_1$-exact if and only if the graph $G$ is perfect. Recall that a
subset of $V$ is {\em stable} in $G$ if it contains no edge. The {\em
  stable set ideal} of $G$ is the vanishing ideal of the $0/1$
characteristic vectors of the stable sets in $G$ and is generated by
the binomials $x_i^2-x_i$ ($i\in V$) and $x_ix_j$ ($\{i,j\}\in E$)
(cf. \cite{Lovasz} for details). 

For a graph $G$, let $IG$ be the vanishing ideal of the incidence
vectors of cuts in $G$, and the cut polytope, $\CUT(G)$, be the convex
hull of the incidence vectors of cuts in $G$.  Following Problem~8.3,
Problem~8.4 in \cite{Lovasz} asks for a characterization of
``cut-perfect'' graphs which are precisely those graphs $G$ for which
$IG$ is $\TH_1$-exact. We answer this question
(Corollary~\ref{corcutperfect}) by studying theta bodies in the more
general setting of cycles in binary matroids. As an intermediate step
we derive the theta bodies of $IG$ which give rise to a new hierarchy
of semidefinite programming relaxations of $\CUT(G)$.

\subsection*{Some notation}
Let $E$ be a finite set. For a subset $F\subseteq E$, let $\one^F\in
\{0,1\}^E$ denote its 0/1-incidence vector and $\chi^F\in\{\pm 1\}^E$
its $\pm 1$-incidence vector, defined by $\one^F_e=1$, $\chi^F_e =-1$
if $e\in F$ and $\one^F_e=0$, $\chi^F_e =1$ otherwise.  Throughout
$\RR E:=\RR[x_e \mid e\in E]$ denotes the polynomial ring with
variables indexed by $E$. If $F \subseteq E$, we set $\xx^F :=
\prod_{e \in F} x_e$. For a symmetric matrix $X \in \RR^{n \times n}$,
$X\succeq 0$ means that $X$ is positive semidefinite, or equivalently,
$\uu^TX\uu\ge 0$ for all $\uu\in\RR^n$.

\subsection*{Contents of the paper} 
% In Section~\ref{secprelims} we recall various preliminaries and the
% results from \cite{GPT} that are needed in this paper.
Section ~\ref{secprelims} contains various preliminaries and some
results of \cite{GPT} needed in this paper.
% The set of cuts (cycles) in a graph form a binary matroid, and the
% language of matroids allows more efficient proofs of our results for
% cuts, while simultaneously being more general.
In Section~\ref{secmatroid} we introduce binary matroids, which
provide the natural setting to present our results for cuts in graphs.
A binary matroid is a pair $\M = (E,\C)$ where $E$ is a finite set and
$\C$ is a collection of subsets of $E$ (the {\em cycles} of $\M$)
closed under taking symmetric differences; for instance, cuts (resp.,
cycles) in a graph form binary matroids.  In Section
\ref{seccycleideals} we present a generating set for the cycle ideal
$I\M$ (i.e. the vanishing ideal of the incidence vectors of the cycles
$C\in\C$)
% in the polynomial ring $\RR E$ with variables $x_e$ ($e\in E$))
and a linear basis $\B$ of its quotient space $\RR E/I\M$ (cf. Theorem
\ref{theobaseM}).  Using this, we can explicitly describe the series
of theta bodies $\TH_k(I\M)$ that approximate the cycle polytope
$\CYC(\M)$ (i.e. the convex hull of the incidence vectors of the
cycles in $\C$). In Section~\ref{secapplicut}, we specialize these
results to cuts in a graph $G$
%(then $I\M$ becomes the cut ideal $IG$)  
and show that $\B$ can then be indexed by $T$-joins of $G$. This
enables a combinatorial description of the theta bodies $\TH_k(IG)$
that converge to the cut polytope $\CUT(G)$ of
$G$. Section~\ref{seccompare} compares the semidefinite relaxations
$\TH_k(IG)$ to some known semidefinite relaxations of the cut
polytope. In Section~\ref{secapplicircuit} the results from
Section~\ref{seccycleideals} are specialized to cycles in a graph.
%%%%%
Section \ref{seccomplexity} contains a discussion about the complexity of constructing theta bodies.
Section~\ref{secexact} studies the binary matroids $\M$ whose cycle
ideal $I\M$ is $\TH_1$-exact (i.e., $\TH_1(I\M) = \CYC(\M)$).
Theorem~\ref{theochar} characterizes the $\TH_1$-exact cycle ideals
$I\M$ when $\M$ does not have the three special minors $F_7^*$,
$R_{10}$ and $\M_{K_5}^*$. As an application, we obtain
characterizations of $\TH_1$-exact graphic and cographic matroids, and
the latter answers Problem~8.4 in \cite{Lovasz}. The paper contains
several examples of binary matroids for which we exhibit the least $k$
for which $I\M$ is $\TH_k$-exact. In Section~\ref{seccircuit} we do
this computation for an infinite family of graphs; if $C_n$ is the
circuit with $n$ edges, then the smallest $k$ for which $\TH_k(IC_n) =
\CUT(C_n)$ is $k = \lceil n/4 \rceil$.

\section{Preliminaries} 
\label{secprelims}

\subsection{Ideals and combinatorial moment matrices}
Let $\RR[\xx]$ be the polynomial ring over $\RR$ in the variables
$\xx=(x_1,\ldots,x_n)$.  A non-empty subset $I \subseteq \RR[\xx]$ is
an {\em ideal} if
%$0 \in I$, 
$I$ is closed under addition, and multiplication by elements of
$\RR[\xx]$.  The ideal generated by $\{f_1,\ldots,f_s\}\subseteq
\RR[\xx]$ is the set $I=\{ \sum_{i=1}^{s} h_if_i \,:\, h_i \in
\RR[\xx] \}$, denoted as $I = (f_1, \ldots, f_s)$. For $S\subseteq
\RR^n$, the {\em vanishing ideal} of $S$ is $\I(S):=\{f\in\RR[\xx]\mid
f(\xx)=0 \ \forall \xx\in S\}.$ For $W \subseteq [n]$, $I_W := I\cap
\RR[x_i\mid i\in W]$ is the {\em elimination ideal} of $I$ with
respect to $W$.

An ideal $I\subseteq \RR[\xx]$ is said to be \emph{zero-dimensional}
if its (complex) variety:
$$V_\CC(I):=\{\xx\in \CC^n\mid f(x)=0\ \; \forall f\in I\},$$ is finite,
$I$ is {\em radical} if $f^m\in I$ implies $f\in I$ for any $f\in
\RR[\xx]$, and $I$ is {\em real radical} if
$f^{2m}+\sum_{i=1}^tg_i^2\in I$ implies $f\in I$ for all
$f,g_i\in\RR[\xx]$. By the {\em Real Nullstellensatz}
(cf. \cite{BCR}), $I$ is real radical if and only if $I =
\I(V_{\RR}(I))$. Therefore, $I$ is zero-dimensional and real radical
if and only if $I = \I(S)$ for a finite set $S \subseteq \RR^n$.  If
$I$ is real radical, and $\pi_W$ denotes the projection from
$\RR^{[n]}$ to $\RR^W$, then the elimination ideal $I_W$ is the
vanishing ideal of $\pi_W(V_\RR(I))$, and there is a simple
relationship between the $k$-th theta body of $I$ and that of its
elimination ideal $I_W$:
\begin{equation}
\label{incproj}
\pi_W(\TH_k(I))\subseteq \TH_k(I_W).
\end{equation}

% The quotient space $\RR[\xx]/I$ is a $\RR$-vector space whose
% elements, called the {\em cosets} of $I$, are denoted as $f+I$
% ($f\in\RR[\xx]$). For $f,g \in \RR[\xx]$, $f+I = g+I$ if and only if
% $f-g \in I$, equivalently, $f$ is congruent to $g$ modulo $I$. The
% degree of $f+I$ is defined as the smallest possible degree of $g\in
% \RR[\xx]$ such that $f-g\in I$. Let $\B$ denote a linear basis of
% $\RR[\xx]/I$. The vector space $\RR[\xx]/I$ has finite dimension if
% and only if $I$ is zero-dimensional. In this case, $\B$ is finite and
% $|V_\CC(I)|\le \dim \RR[\xx]/I$, with equality if and only if $I$ is 
% radical.
The quotient space $\RR[\xx]/I$ is a $\RR$-vector space whose
elements, called the {\em cosets} of $I$, are denoted as $f+I$
($f\in\RR[\xx]$). For $f,g \in \RR[\xx]$, $f+I = g+I$ if and only if
$f-g \in I$.  The degree of $f+I$ is defined as the smallest possible
degree of $g\in \RR[\xx]$ such that $f-g\in I$.  The vector space
$\RR[\xx]/I$ has finite dimension if and only if $I$ is
zero-dimensional; moreover, $|V_\CC(I)|\le \dim \RR[\xx]/I$, with
equality if and only if $I$ is radical.

\medskip
Gouveia et al. \cite{GPT} give a geometric characterization of
zero-dimensional real radical ideals that are $\TH_1$-exact. 

\begin{definition} \label{def:level}
For $k\in\NN$, a finite
set $S\subseteq \oR^n$ is said to be {\em $k$-level} if $|\{f(\xx)\mid
\xx\in S\}|\le k$ for all $f\in \RR[\xx]_1$ for which the linear
inequality $f(\xx)\ge 0$ induces a facet of the polytope $\conv(S)$.
\end{definition}

\begin{theorem}\cite{GPT}\label{theo2level}
  Let $S\subseteq \oR^n$ be a finite set. The ideal $\I(S)$ is
  $\TH_1$-exact (i.e., $\conv(S)=\TH_1(\I(S))$) if and only if $S$ is
  a $2$-level set.
\end{theorem}

\noindent
More generally, Gouveia et al. \cite[Section 4]{GPT} show the implication:
\begin{equation}\label{relklevel}
S \text{ is } (k+1)\text{-level } \Longrightarrow  \;
\I(S) \text{ is } \TH_k\text{-exact};
\end{equation}
the reverse implication however does not hold for $k\ge 2$ (see
e.g.~Remark~\ref{remcircuit} for a counterexample).

\medskip
We now mention an alternative more explicit formulation for the theta
body $\TH_k(I)$ of an ideal $I$ in terms of positive semidefinite
combinatorial moment matrices.  We first recall this class of matrices
(introduced in \cite{Lau07a}) which amounts to using the equations
defining $I$ to reduce the number of variables.
% Define the degree of a coset $f+I$ as the minimum degree of
% $f'\in\RR[\xx]$ defining the same coset (i.e., $f-f'\in I$).
Let $\B=\{b_0+I,b_1+I, \ldots\}$ be a basis of $\RR[\xx]/I$ and, for 
$k\in \NN$, let $\B_k:=\{b+I\in\B\mid \deg(b+I)\le k\}$.
% $ denote the subset of $\B$ consisting of its elements of degree at
% most $k$.
Then any polynomial $f\in \RR[\xx]$ has a unique decomposition $f
=\sum_{l\ge 0} \lambda^{(f)}_l b_l$ modulo $I$; we let
$\lambda^{(f)}=(\lambda^{(f)}_l)_l$ denote the vector of coordinates
of the coset $f+I$ in the basis $\B$ (which has only finitely many
non-zero coordinates).

\begin{definition}
  Let $\yy\in\RR^\B$. The {\em combinatorial moment matrix}
  $M_\B(\yy)$ is the (possibly infinite) matrix indexed by $\B$ whose
  $(i,j)$-th entry is
$$\sum_{l\ge 0} \lambda^{(b_ib_j)}_l y_l.$$
The $k$-th truncated combinatorial moment matrix $M_{\B_k}(\yy)$ is
the principal submatrix of $M_\B(\yy)$ indexed by $\B_k$.
\end{definition}

In other words, the matrix $M_\B(\yy)$ is obtained as follows. The
coordinates $y_l$'s correspond to the elements $b_l+I$ of $\B$; expand
the product $b_ib_j$ in terms of the basis $\B$ as $b_ib_j=\sum_l
\lambda^{(b_ib_j)}_l b_l$ modulo $I$; then the $(b_i,b_j)$-th entry of
$M_\B(\yy)$ is its `linearization': $\sum_l \lambda^{(b_ib_j)}_l y_l$.

To control which entries of $\yy$ are involved in the truncated matrix
$M_{\B_k}(\yy)$, it is useful to suitably choose the basis
$\B$. Namely, we choose $\B$ satisfying the following property:
\begin{equation}\label{propB}
\deg(f+I)\le k\Longrightarrow f+I \in\span(\B_k).
\end{equation}
This is true, for instance, when $\B$ is the set of {\em standard
  monomials} of a {\em term order} that respects degree. (See
\cite[Chapter 2]{CLO} for these notions that come from Gr\"obner basis
theory.) If $\B$ satisfies (\ref{propB}), then the entries of
$M_{\B_k}(\yy)$ depend only on the entries of $\yy$ indexed by
$\B_{2k}$.  Moreover, Gouveia et al.  \cite{GPT} show that $\TH_k(I)$
can then be defined using the matrices $M_{\B_k}(\yy)$, up to closure
and a technical condition on $\B$.  This technical condition, which
states that $\{1+I,x_1+I,\ldots,x_n+I\}$ is linearly independent in
$\RR[\xx]/I$, is however quite mild since if there is a linear
dependency then it can be used to eliminate variables.

\begin{example} 
\label{ex:runningex2} 
Consider the ideal $I = (x_1^2x_2 - 1) \subset \RR[x_1,x_2]$. Note
that $\B = \bigcup_{k \in \NN} \{x_1^k+I, x_2^k+I, x_1x_2^k+I\}$ is a
monomial basis for $\RR[x_1,x_2]/I$ satisfying (\ref{propB}) for
which $$\B_4 =
\{1,x_1,x_2,x_1^2,x_1x_2,x_2^2,x_1^3,x_1x_2^2,x_2^3,x_1^4,x_1x_2^3,x_2^4
\}+I.$$ The combinatorial moment matrix $M_{\B_2}(\yy)$ for
$\yy=(y_0,y_1,\ldots, y_{11}) \in \RR^{\B_4}$ is
  $$ \begin{array}{c||c|c|c|c|c|c|}
    & 1 & x_1 & x_2 & x_1^2 & x_1x_2 & x_2^2 \\
\hline
\hline
  1 & y_0 & y_1 & y_2 & y_3 & y_4 & y_5 \\
\hline
  x_1 & y_1 & y_3 & y_4 & y_6 & 1 & y_7 \\
\hline
  x_2 & y_2 & y_4 & y_5 & 1 & y_7 & y_8 \\
\hline
 x_1^2& y_3 & y_6 & 1 & y_9 & y_1 & y_2 \\
\hline
 x_1x_2 & y_4 & 1 & y_7 & y_1 & y_2 & y_{10}\\
\hline
 x_2^2& y_5 & y_7 & y_8 & y_2 & y_{10} & y_{11}\\
\hline 
  \end{array}.$$
\end{example}

\vspace{.2cm}

\begin{theorem}\cite{GPT} \label{thepTHMB}
  Assume $\B$ satisfies (\ref{propB}) and
  $\B_1=\{1+I,x_1+I,\ldots,x_n+I\}$, and let the coordinates of
  $\yy\in\RR^{\B_{2k}}$ indexed by $\B_1$ be
  $y_0,y_1,\ldots,y_n$. Then $\TH_k(I)$ is equal to the closure of the
  set
\begin{equation}\label{setMB}
  \{(y_1,\ldots,y_n)\mid \yy\in\RR^{\B_{2k}} \text{ with }
  M_{\B_k}(\yy)\succeq 0  \text{ and }  y_0=1\}. 
\end{equation}
When $I=\I(S)$ where $S\subseteq \{0,1\}^n$, the closure is not needed
and $\TH_k(I)$ equals the set (\ref{setMB}).
\end{theorem}

Theorem~\ref{thepTHMB} implies that optimizing a linear objective
function over $\TH_k(I)$ can be reformulated as a semidefinite program
with the constraints $M_{\B_k}(\yy)\succeq 0$ and $y_0=1$ which, for
fixed $k$, can thus be solved in polynomial time (to any precision).

\subsection{Graphs, cuts and cycles}

Let $G=(V,E)$ be a graph.  Throughout, the vertex set is $V=[n]$, the
edge set of the complete graph $K_n$ is denoted by $E_n$, so that $E$
is a subset of $E_n$, and the edges of $E_n$ correspond to pairs
$\{i,j\}$ of distinct vertices $i,j\in V$.  For $F\subseteq E$,
$\deg_F(v)$ denotes the number of edges of $F$ incident to $v\in V$.
A {\em circuit} is a set of edges
$\{\{i_1,i_2\},\{i_2,i_3\},\ldots,\{i_{t-1},i_t\},\{i_t,i_1\}\}$ where
$i_1,\ldots,i_t\in V$ are pairwise distinct vertices.  A set
$C\subseteq E$ is a {\em cycle} (or Eulerian subgraph) if $\deg_C(v)$
is even for all $v\in V$;
% A circuit is an inclusion-wise minimal non-empty cycle.
every non-empty cycle is an edge-disjoint union of circuits.  For
$S\subseteq V$, the {\em cut} $D$ corresponding to the partition
$(S,V\setminus S)$ of $V$ is the set of edges $\{i,j\}\in E$ with
$|\{i,j\}\cap S|=1$. A basic property is that each cut intersects each
cycle in an even number of edges; this is in fact a property of binary
matroids which is why we will present some of our results later in the
more general setting of binary matroids (cf. Section
\ref{secmatroid}).

Each cut $D$ can be encoded by its $\pm 1$-incidence vector
$\chi^D\in\{\pm 1\}^E$, called the {\em cut vector} of $D$. The {\em
  cut ideal} of $G$, denoted as $IG$, is the vanishing ideal of the
set of cut vectors of $G$.
% with $(\chi_D)_e=-1$ if $e\in D$ and $(\chi_D)_e=1$ otherwise.
The {\em cut polytope} of $G$ is 
\begin{equation}\label{cutP}
  \CUT(G):=\conv\{\chi^D\mid D \ \text{ is a cut in } G\} =
  \pi_E(\CUT(K_n)) \subseteq \RR^E, 
\end{equation}
where $\pi_E$ is the projection from $\RR^{E_n}$ onto
$\RR^E$. (Cf. e.g. \cite{DL97} for an overview on the
cut polytope.)  The cuts of $K_n$ can also be encoded by the {\em cut
  matrices} $X:= \xx\xx^T$ for $\xx \in \{\pm 1\}^n$ indexing the 
partitions of $[n]$ corresponding to the cuts. Thus the set
\begin{equation}\label{sdpGW}
%\{X\in \RR^{V\times V}\mid X\succeq 0,\ X_{ii}=1 \ \forall i\in V\}
\{\yy\in \RR^{E}\mid \exists X\in \RR^{V\times V},
X\succeq 0, \ X_{ii}=1\ (i\in V), \ X_{ij}=y_{\{i,j\}}\ (\{i,j\}\in E)\}
\end{equation}
is a relaxation of the cut polytope $\CUT(G)$, over which one can
optimize any linear objective function in polynomial time (to any
precision), using semidefinite optimization.
% (since the upper triangular parts of $Y$ contains precisely the cut
% vectors of $K_n$).  

Given edge weights $\ww \in \RR^E$, the max-cut problem asks for a cut $D$
in $G$ of maximum total weight $\sum_{e\in D}w_{e}$; thus it can be
formulated as
\begin{equation}\label{maxcut}
 \max \left\{ \frac{1}{2}\sum_{e \in E} w_{e}(1- y_{e}) 
\mid  \yy \in \textup{CUT}(G)\right\},
\end{equation} 
where the variable can alternatively be assumed to lie in $\CUT(K_n)$.
This is a well-known NP-hard problem \cite{GJ79}. Thus one is
interested in finding tight efficient relaxations of the cut polytope,
potentially leading to good approximations for the max-cut problem.
It turns out that the simple semidefinite programming relaxation
(\ref{sdpGW}) has led to the celebrated $0.878$-approximation
algorithm of Goemans and Williamson \cite{GW95} which, as of today,
still gives the best known performance guarantee for max-cut.

\section{Theta bodies for cuts and matroids} \label{secmatroid}

In this section we study in detail the hierarchy of SDP relaxations
for the cut polytope arising from the theta bodies of the cut ideal.
As is well-known, cuts in graphs form a special class of binary
matroids.  It is thus natural to consider the theta bodies in the more
general setting of binary matroids, where the results become more
transparent.  Then we will apply the results to cuts in graphs (the
case of cographic matroids) and also to cycles in graphs (the case of
graphic matroids).

\subsection{The cycle ideal of a binary matroid and its theta bodies} 
\label{seccycleideals}

Let $\M=(E,\C)$ be a binary matroid; that is, $E$ is a finite set and
$\C$ is a collection of subsets of $E$ that is closed under taking
symmetric differences.  Members of $\C$ are called the {\em cycles} of
$\M$, and members of the set
$$\C^*:=\{D\subseteq E :  |D\cap C| \text{ even } \ \forall C\in\C\}$$
are called the {\em cocycles} of $\M$.  Then, $\M^*=(E,\C^*)$ is again
a binary matroid, known as the {\em dual matroid} of $\M$, and
$(\M^*)^*=\M$.  The (inclusion-wise) minimal non-empty cycles
(cocycles) of $\M$ are called the {\em circuits} ({\em cocircuits}) of
$\M$.  An element $e\in E$ is a {\em loop} ({\em coloop}) of $\M$ if
$\{e\}$ is a circuit (cocircuit) of $\M$.  Two distinct elements
$e,f\in E$ are {\em parallel} ({\em coparallel}) if $\{e,f\}$ is a
circuit (cocircuit) of $\M$.  Every non-empty cycle is a disjoint
union of circuits.  Given $C\in \C$, an element $e\in E\setminus C$ is
called a {\em chord} of $C$ if there exist $C_1, C_2\in\C$ such that
$C_1\cap C_2=\{e\}$ and $C=C_1\Delta C_2$ (if $C$ is a circuit then
$C_1,C_2$ are in fact circuits); $C$ is said to be {\em chordless} if
it has no chord.  Here is a property of chords that we will use later.

\begin{lemma}\label{lemchord}
  Let $C$ be a circuit of $\M$, let $e\in E\setminus C$ be a chord of
  $C$ and $C_1,C_2$ be circuits with $C=C_1\Delta C_2$ and $C_1\cap
  C_2=\{e\}$.  Then each $C_i$ has strictly fewer chords than $C$.
\end{lemma}

\begin{proof}
  It suffices to show that each chord $e'$ of $C_1$ is also a chord of
  $C$. For this let $C_1',C_1''$ be two circuits with $C_1'\cap
  C_1''=\{e'\}$ and $C_1=C_1'\Delta C_1''$. Say, $e\in C_1'$, and thus
  $e\not\in C_1''$.  Suppose first that $e'\in C_2$. Then we have
  $C_1''\cap C_2=\{e'\}$ and $C_1''\Delta C_2\subseteq C$. As $C$ is a
  circuit and $C_1''\neq C_2$, we deduce that $C=C_1''\Delta C_2$,
  which shows that $e'$ is a chord of $C$.

  Suppose now that $e'\not\in C_2$.  Then, $C=C_1\Delta
  C_2=(C_1'\Delta C_2)\Delta C_1''$ with $(C_1'\Delta C_2)\cap
  C_1''=\{e'\}$, which shows again that $e'$ is a chord of $C$.
%Set $\tilde C:= C_1'\Delta C_2$.
%  Then $\tilde C\cap C_1''=\{e'\}$ and $\tilde C\cap C_1''\subseteq
%  C$. This implies $C=\tilde C\cap C_1''$, and thus $e'$ is a chord of $C$.
\end{proof}

The binary matroids on $E$ correspond to the $\GF(2)$-vector subspaces
of $\GF(2)^E$, where $\GF(2)$ is the two-element field $\{0,1\}$ with
addition modulo 2.  Namely, identifying a set $F \subseteq E$ with its
0/1-incidence vector $\one^F\in \GF(2)^E$, the set of cycles $\C$ is a
vector subspace of $\GF(2)^E$ and the set of cocycles $\C^*$ is its
orthogonal complement. Thus the cycles of a binary matroid also arise
as the solutions in $\GF(2)^E$ of a linear system $M\xx =0$, where $M$
is a matrix with columns indexed by $E$, called a {\em representation
  matrix} of the matroid.  In what follows we will use $\C$ (and
$\C^*$) both as a collection of subsets of $E$ and as a $GF(2)$-vector
space.

\medskip
As before let $\RR E:=\RR[x_e \mid e\in E]$ and, for  $C \in \C$, let 
$ \chi^C \in
\{\pm 1\}^E$ denote its $\pm 1$-incidence vector, called its {\em
  cycle vector}. 
% the vector with $(\chi_C)_e = -1$ if $e \in C$ and $(\chi_C)_e = 1$
% if $e \not \in C$.
Then,
  $$\CYC(\M):=\conv(\chi^C\mid C\in \C)$$
is  the {\em cycle polytope} of $\M$ and  
$$I\M:= \I(\chi^C\mid C\in \C)$$
is the vanishing ideal of the cycle vectors of $\M$, called the {\em
  cycle ideal} of $\M$.  Thus $I\M$ is a real radical zero-dimensional
ideal in $\RR E$.
%For instance, when $\M$ is the cographic matroid of a graph $G$, then 
%$\CYC(\M)$ is the cut polytope and $I\M$ is the cut ideal of $G$.

We first study the quotient space $\RR E/I\M$. For this consider the set
% let us introduce the following set of binomials:
\begin{equation}\label{setH}
%\begin{array}{c}
%\H:=\{x_e^2-1\ \mid  e\in E\} \\ \bigcup 
%\{\xx^A-\xx^B \mid  A,B\subseteq E \text{ disjoint,} \ A\cup B 
%\text{ chordless cocircuit of }\M\}. 
%\end{array}
\H:=\{ x_e^2-1\  (e\in E),\ 1-\xx^{D}\ (D \text{ chordless cocircuit of } \M)\}.
\end{equation}
Obviously, $\H\subseteq I\M$; Theorem \ref{theobaseM} below shows that
$\H$ in fact generates the ideal $I\M$.  First we observe that $\H$
also generates all binomials $\xx^A-\xx^B$ where $A\cup B$ partitions
any cocycle of $\M$.

\begin{lemma}\label{lemH}
Let $D\in\C^*$ be partitioned as $D=A\cup B$. Then, 
$\xx^A-\xx^B \in (\H)$.
\end{lemma}

\begin{proof}
  First we note that it suffices to show that $1-\xx^D\in (\H)$ for
  all $D\in \C^*$. Indeed, for any partition $A\cup B=D$, $\
  \xx^A(1-\xx^D)=\xx^A -(\xx^A)^2\xx^B \equiv \xx^A-\xx^B$ modulo
  $(\H)$.  Thus $1-\xx^D\in (\H)$ implies $\xx^A-\xx^B \in (\H)$.

  Next, we show the lemma for the case when $D$ is a cocircuit, using
  induction on the number $p$ of its chords. If $p=0$ then $1-\xx^D\in
  \H$ by definition. So let $p\ge 1$, let $e$ be a chord of $D$ and
  let $D_1,D_2$ be cocircuits with $D=D_1\Delta D_2$ and $D_1\cap
  D_2=\{e\}$. Then, $1-\xx^{D_1}, 1-\xx^{D_2}\in (\H)$, using the
  induction assumption, since each $D_i$ has at most $p-1$ chords by
  Lemma \ref{lemchord}.  We have: $1-\xx^D\equiv 1-(x_e)^2
  \xx^{D_1\setminus \{e\}}\xx^{D_2\setminus \{e\}} = 1- \xx^{D_1}
  \xx^{D_2} = \xx^{D_1}(1-\xx^{D_2}) + 1-\xx^{D_1}$, where the first
  equality is modulo $(\H)$. This shows that $1-\xx^D\in (\H)$.
  % Set $A_i:=A\cap D_i$, $B_i:=B\cap D_i$; thus $D_i$ is partitioned
  % as $A_i\cup B_i\cup \{e\}$.  Then we have: $ \xx^A-\xx^B \equiv
  % \xx^{A_1}\xx^{A_2}x_e^2 -\xx^{B_1}\xx^{B_2} =
  % \xx^{A_1}x_e(\xx^{A_2}x_e-\xx^{B_2}) + \xx^{B_2}(\xx^{A_1}x_e
  % -\xx^{B_1})$ (where the first equality is modulo the ideal
  % $(\H)$).  In view of Lemma \ref{lemp}, we can apply the induction
  % assumption and conclude that $\xx^{A_i}x_e-\xx^{B_i}\in (\H)$,
  % which thus implies $\xx^A-\xx^B\in (\H)$.

  Finally we show the lemma for $D\in\C^*$, using induction on the
  number $p$ of cocircuits in a partition of $D$.  For this, let
  $D=D_1\cup D_2$, where $D_1$ is a cocircuit and $D_2$ is a cocycle
  partitioned into $p-1$ cocircuits. Then, by the previous case,
  $1-\xx^{D_1}\in (\H)$, and $1-\xx^{D_2}\in (\H)$ by the induction
  assumption.  Then, $1-\xx^D \equiv (\xx^{D_2})^2- \xx^{D_1}\xx^{D_2}
  = \xx^{D_2}(1-\xx^{D_1})-\xx^{D_2}(1-\xx^{D_2})$, where the first
  equality is modulo $(\H)$. This implies $1-\xx^D\in (\H)$.
  % Set $A_i:=A\cap D_i$, $B_i=B\cap D_i$.  Then, $x^A-\xx^B =
  % \xx^{A_1}\xx^{A_2}-\xx^{B_1}\xx^{B_2} =
  % \xx^{A_2}(\xx^{A_1}-\xx^{B_1}) + \xx^{B_1}(\xx^{A_2}-\xx^{B_2})\in
  % (\H)$, since $\xx^{A_1}-\xx^{B_1}\in (\H)$ (by the case above, as
  % $D_1$ is a cocircuit) and $\xx^{A_2}-\xx^{B_2}\in (\H)$ (by the
  % induction assumption).
\end{proof}

Define the relation `$\sim$' on $\PPP(E)$, the collection of all
subsets of $E$, by
\begin{equation}\label{equirel}
F\sim F' \ \text{ if }  F\Delta F'\in \C^*;
\end{equation}
this is an equivalence relation, since $\C^*$ is closed under taking
symmetric differences.  The next lemma characterizes the equivalence
classes.

\begin{lemma}\label{lemequi}
  For $F,F'\subseteq E$, we have:
%the following assertions are equivalent:
$$F\Delta F'\in \C^* \Longleftrightarrow \xx^F-\xx^{F'}\in (\H) 
\Longleftrightarrow \xx^F-\xx^{F'}\in I\M.$$
%(i)  $F\Delta F'\in \C^*$; (ii) $\xx^F-\xx^{F'}\in (\H)$; and (iii)
%$\xx^F-\xx^{F'}\in I\M$. 
\end{lemma}

\begin{proof}
  If $F\Delta F'\in \C^*$, then
%For (i) $\Longrightarrow $ (ii), use the fact that
  $\xx^F-\xx^{F'}=\xx^{F\cap F'}(\xx^{F\setminus F'}-\xx^{F'\setminus
    F})\in (\H)$, using Lemma \ref{lemH}; $\xx^F-\xx^{F'}\in (\H)
  \Longrightarrow \xx^F-\xx^{F'}\in I\M$ follows
%This shows (i) $\Longrightarrow$ (ii),
% and (ii)  $\Longrightarrow $ (iii) follows
  from $\H\subseteq I\M$.  Conversely, if $\xx^F-\xx^{F'}\in I\M$
  then, for any $C\in\C$, $\xx^F-\xx^{F'}$ vanishes at $\chi^C$ and
  thus $|C\cap F|$ and $|C\cap F'|$ have the same parity, which
  implies that $|C\cap (F\Delta F')|$ is even and thus $F\Delta
  F'\in\C^*$.
\end{proof}

Let 
\begin{equation}\label{setF}
\F:=\{F_1,\ldots,F_N\}
\end{equation}
be a set of distinct representatives of the equivalence classes of
$\PPP(E)/\sim$ and set
\begin{equation}\label{setB}
\B:=\{\xx^F +I\M \mid F\in\F\}.
\end{equation}

\begin{theorem}
\label{theobaseM}
The set $\B$ is a basis of the vector space $\RR E/I\M$ and the set
$\H$ generates the ideal $I\M$.
\end{theorem}

\begin{proof}
  First, we show that $\B$ spans the space $\RR E/(\H)$. As
  $x_e^2-1\in \H$ ($\forall e\in E$), it suffices to show that $\B$
  spans all cosets of square-free monomials.  For this, let
  $F\subseteq E$ and, say, $F\sim F_1$; then, $\xx^F-\xx^{F_1}\in I\M$
  by Lemma~\ref{lemequi}, which shows that $\xx^F+I\M \in \span(\B)$.
%Therefore, $\dim \RR E/(\H)\le |\B|=N$. Moreover, as $\H\subseteq I\M$,
%$\dim \RR E/I\M \le \dim \RR E/(\H)$, and 
%$\dim \RR E/I\M  = |\C|$. Summarizing,
Therefore, we obtain:
$$|\C|=\dim \RR E/I\M\le \dim \RR E/(\H)\le |\B|=N.$$
To conclude the proof it now suffices to show that $|\C|=N$.  For
this, fix a basis $\{C_1,\ldots,C_m\}$ of the $\GF(2)$-vector space
$\C$, so that $|\C|=2^m$.
% and the cycles are all sets $\Delta_{i\in I}C_i$ for $I\subseteq
% [m]$.
Let $M$ be the $m\times |E|$ matrix whose rows are the 0/1-incidence
vectors of $C_1,\ldots,C_m$. Then $M\xx$ takes $2^m$ distinct values
for all $\xx\in\GF(2)^E$. As, for $F,F'\subseteq E$, $F\sim F'$ if and
only if $M\one^F =M\one^{F'}$, we deduce that the equivalence relation
(\ref{equirel}) has $N=2^m$ equivalence classes.
\end{proof}

\medskip 
%%Changes here
We now consider the combinatorial moment matrices for the
cycle ideal $I\M$.  
%To obtain a basis $\B$ of $\RR E/I\M$ satisfying
%(\ref{propB}), we choose in each equivalence class of (\ref{equirel})
%a representative of minimum cardinality; that is, each $F_i$ in the
%set $\F$ in (\ref{setF}) satisfies:
%\begin{equation}\label{relcard}
%|F_i|\le F \ \text{ for all } \  F\subseteq E \ \text{ with }
%F_i\Delta F\in \C^*. 
%\end{equation}
For any integer $k$ define the set 
\begin{equation}\label{relFk}
\F_k:=\{F\in\F\mid \exists D\in\C^*\,\textup{with}\, |F\Delta D|\le k\}
\end{equation}
corresponding to the equivalence classes of $\sim$ having a
representative of cardinality at most $k$.  Then $\B_k=\{\xx^F+I\M
\mid F\in \F_k\}$ can be identified with the set $\F_k$. Moreover
relation (\ref{propB}) holds, so that the entries of the truncated
moment matrix $M_{\B_k}(\yy)$ depend only on the entries of $\yy$
indexed by $\B_{2k}$.  For instance, $\F_1$ can be any maximal subset
of $E$ containing no coloops or coparallel elements of $\M$, along
with $\emptyset$.  Indeed, $e\in E$ is a coloop precisely if $\{e\}
\sim \emptyset$, and two elements $e\ne f\in E$ are coparallel
precisely if $e \sim f$.  Thus, $\F_0=\{\emptyset\}$ and
$\F_1\setminus \F_0=E$ if $\M$ has no coloops and no coparallel
elements.
%If $F_i, F_j \in \F$ and if $F_i\Delta F_j\sim F_l$ for $F_l\in \F$,
%then $|F_l|\le |F_i\Delta F_j|$; thus $F_i, F_j\in \F_k$ implies $F_l
%\in\F_{2k}$.  Therefore, for $\yy =(y_{F})_{F\in \F_{2k}}\in
%\RR^{\B_{2k}}$, the combinatorial moment matrix $M_{\B_k}(\yy)$, which
%is the matrix indexed by $\B_k$ whose $(F_i,F_j)$-th entry is
%$y_{F_l}$, is well defined.

When $\M$ has no coloops and no coparallel elements, its $k$-th theta
body $\TH_k(I\M)$ consists of the vectors $\yy\in \RR^E$ for which
there exists a positive semidefinite $|\F_k|\times |\F_k|$ matrix $X$
satisfying $X_{\emptyset,e} = y_e$ for all $e\in E$ and
\begin{equation}\label{matX}
\begin{array}{l}
%\text{(i) } y_e = X_{\emptyset,e} \text{ for all } e\in E;\\
\text{(i) }\  X_{\emptyset,\emptyset}=1,\\
\text{(ii) } X_{F_1,F_2}=X_{F_3,F_4} \text{ if }  
F_1\Delta F_2 \Delta F_3\Delta F_4\in \C^*.
\end{array}
\end{equation}

\begin{remark}\label{remconstraint}
  The constraints (\ref{matX})(ii) contain in particular the
  constraints
\begin{equation}\label{matXmom}
  X_{F_1,F_2}=X_{F_3,F_4} \text{ if }  F_1\Delta F_2 = F_3\Delta F_4.
\end{equation}
Note that the above constraints are the basic `moment constraints',
which are satisfied by all $\pm 1$ vectors.  Indeed, if $\yy=\chi^F\in
\{-1,1\}^E$, define the $|\F_k|\times |\F_k|$ matrix $X$ by
$X_{F_1,F_2}:= (-1)^{|F\cap F_1|} (-1)^{|F \cap F_2)|}$, so that
$y_e=X_{\emptyset,e}$ ($e\in E$). Then $X \succeq 0$ since $X=\uu
\uu^T$ where $\uu=((-1)^{|F\cap F_i|})_{F_i\in \F_k}$, and $X$
satisfies (\ref{matX})(i) and (\ref{matXmom}).  Therefore the
constraints (\ref{matXmom}) do not cut off any point of the cube
$[-1,1]^E$.  Non-trivial constraints that cut off points of $[-1,1]^E$
that do not lie in $\CYC(\M)$ come from those constraints 
(\ref{matX})(ii) where $F_1\Delta F_2 \Delta F_3\Delta F_4$ is a
non-empty cocycle.
\end{remark}

\subsection{Application to cuts in graphs}\label{secapplicut}

Binary matroids arise naturally from graphs in the following way.  Let
$G=([n],E)$ be a graph, let $\C_G$ denote its collection of cycles, and
$\D_G$ its collection of cuts.  Since $\C_G$ and $\D_G$ are closed
under symmetric difference, both $\M_G:=(E,\C_G)$ and
$\M_G^*:=(E,\D_G)$ are binary matroids, and since each cut has an even
intersection with each cycle, they are duals of each other. The
matroid $M_G$ is known as the {\em graphic matroid} of $G$ and
$\M_G^*$ as its {\em cographic matroid}.

\medskip

We consider here the case when $\M=\M_G^*$ is the cographic matroid of
$G=([n],E)$.  Then, $\CYC(\M)=\CUT(G)$ is the cut polytope of $G$ and
$I\M$ is the {\em cut ideal} of $G$ (denoted earlier by $IG$), thus
defined as the vanishing ideal of all cut vectors in $G$.
% Our objective in this note is to study the new semidefinite
% programming relaxations of the cut polytope obtained by applying the
% construction of the theta bodies in the previous section to the
% vanishing ideal of cuts. For this, let So $IG$ is an ideal in the
% polynomial ring $\RR E=\RR[x_{\{i,j\}}\mid \{i,j\}\in E]$.  The next
% lemma gives some easy information about the ideals $IG$ and $IK_n$.

So $IG$ is an ideal in $\RR E$, while $I K_n$ is an ideal in $\RR
E_n$.  One can easily verify that $IG$ is the elimination ideal, $I
K_n \cap \RR E$, of $I K_n$ with respect to $E$.  By Theorem
\ref{theobaseM}, we know that the (edge) binomials $x_e^2-1$ ($e\in
E$) together with the binomials $1-\xx^C$ ($C$ chordless circuit of
$G$) generate the cut ideal $IG$. When $G=K_n$ is a complete graph,
the only chordless circuits are the triangles so that, beside the edge
binomials, it suffices to consider the binomials $1-
x_{\{i,j\}}x_{\{i,k\}}x_{\{j,k\}}$ (or
$x_{\{i,j\}}-x_{\{i,k\}}x_{\{j,k\}}$) for distinct $i,j,k\in [n]$.

When $G$ is connected, there are $2^{n-1}$ distinct cuts in $G$
(corresponding to the partitions of $[n]$ into two classes) and, when
$G$ has $p$ connected components, there are $2^{n-p}$ cuts in $G$ and
thus $\dim \RR E /IG =2^{n-p}$.

%Note that the polynomials in (\ref{cutidealgene}) are precisely the 
%the $2\times 2$ minors of the matrix 
%$$\left[\begin{array}{ccccc}
%     1 & x_{\{1,2\}} & x_{\{1,3\}} & \dots & x_{\{1,n\}}\\
%     x_{\{1,2\}} & 1 & x_{\{2,3\}} & \dots  & x_{\{2,n\}}\\
%     x_{\{1,3\}} & x_{\{2,3\}} & 1 & & x_{\{3,n\}} \\
%     \vdots & \vdots & & \ddots & \vdots\\
%     x_{\{1,n\}} & x_{\{2,n\}} & x_{\{3,n\}} & \dots & 1
% \end{array}\right].$$
%We will give in Lemma \ref{lem?} below an explicit set of generators 
%for the cut ideal $IG$.

\medskip 
The following notion of $T$-joins arises naturally when
considering the equivalence relation (\ref{equirel}).
Given a set $T\subseteq [n]$, a set $F\subseteq E$ is called a {\em
  $T$-join} if $T=\{v\in [n] \mid \deg_F(v) \text{ is odd}\}$.
For instance, the $\emptyset$-joins are the cycles of $G$ and, for
$T=\{s,t\}$, the minimum $T$-joins correspond to the shortest $s-t$
paths in $G$. 
%Note that a $T$-join exists only for $|T|$ even.
%if $T \subseteq [n]$ supports a non-trivial
%$T$-join then $|T|$ is even. 
If $F$ is a $T$-join and $F'$ is $T'$-join, then $F\Delta F'$ is a
$(T\Delta T')$-join. In particular, $F\sim F'$, i.e.  $F\Delta F'\in
\C_G$, precisely when $F,F'$ are both $T$-joins for the same
$T\subseteq [n]$.

Thus the equivalence classes of $\sim$ correspond to the members of
the set $\T_G:=\{T\subseteq [n] \mid \, \exists T\text{-join in } G\}$
(which consists of the sets $T_1\cup\ldots\cup T_p$, where each $T_i$
is an even subset of $V_i$ and $V_1,\ldots,V_p$ are the connected
components of $G$). The set $\F$ (in (\ref{setF})) consists of one
$T$-join $F_T$ for each $T\in \T_G$, and $\F_k=\{F_T \mid T\in \T_k\}$,
after defining $\T_k$ as the set of all $T\in \T_G$ for which there
exists a $T$-join of size at most $k$.
% and (\ref{relcard}) holds when each $F_T$ is chosen to be a {\em
% minimum cardinality} $T$-join. (Recall that a minimum weight
% $T$-join can be computed in polynomial time \cite{Ed65}).   
Then the corresponding basis of $\RR E/IG$ is $\B=\{\xx^{F_T} + IG\mid  
T\in \T_G\}$, $\B_k=\{\xx^{F_T} + IG\mid T\in \T_k\}$ and
(\ref{propB}) holds.

For instance, $\F_1$ consists of all edges
$e\in E$ together with the empty set.
Hence the first order theta body $\TH_1(IG)$ consists of the
vectors $\yy\in \RR^E$ for which there exists a positive semidefinite
matrix $X$ indexed by $E\cup \{\emptyset\}$ satisfying 
$y_{e}=X_{\emptyset,e}$ ($e\in E$) 
and 
\begin{equation}\label{matXG}
\begin{array}{l}
%  \text{(i) }\  \ X_{\emptyset,\emptyset}=X_{\{i,j\},\{i,j\}}=1 \
%  \text{ for  all } \{i,j\}\in E,\\ 
  \text{(i) }\  \ X_{\emptyset,\emptyset}=X_{e, e}=1 \
  \text{ for  all } e\in E,\\ 
%\text{(ii) }\ X_{\{i,j\},\{j,k\}} = X_{\emptyset,\{i,k\}}  \text{ if } 
%(i,j,k) \text{ is a circuit (triangle) in } G,\\
\text{(ii) }\  X_{e,f}=X_{\emptyset,g} \text{ if } \{e,f,g\} \text{
is a triangle in }  G,\\ 
\text{(iii) } X_{e,f}=X_{g,h} \text{ if } \{e,f,g,h\} \text{ is a
circuit in } G. 
%\text{(iii) } X_{\{i,j\},\{h,k\}}=X_{\{i,k\},\{j,h\}} \ \text{ if }  
%(i,j,h,k) \text{ is a circuit in } G.
\end{array}
\end{equation}

\begin{remark}\label{remcomplete}
  When $G=K_n$ is the complete graph, for any even $T\subseteq [n]$,
  the minimum cardinality of a $T$-join is $|T|/2$; just choose for
  $F_T$ a set of $|T|/2$ disjoint edges (i.e. a perfect matching) on
  $T$.  Hence the set $\T_k$ consists of all even $T\subseteq [n]$
  with $|T|\le 2k$.
%(and each $F_T$ consists of $k$ disjoint edges on $T$).  
As an illustration, if
we index the combinatorial moment matrices by $\T_k$, then the
condition (\ref{matX})(ii) reads:
\begin{equation}\label{relTk}
X_{T_1,T_2}=X_{T_3,T_4} \ 
\ \text{ if } \ T_1\Delta T_2=T_3\Delta T_4.
\end{equation}
This observation will enable us to relate the theta body hierarchy to
the semidefinite relaxations of the cut polytope considered in
\cite{LaurentMaxCut}, cf. Section \ref{seccompare}.
\end{remark}

%For instance, the first order theta body $\TH_1(IG)$ consists of the
%vectors $\yy\in \RR^E$ for which there exists a positive semidefinite
%matrix $X$ indexed by $E\cup \{\emptyset\}$ satisfying $y_{\{i,j\}}=
%X_{\emptyset,\{i,j\}}$ ($\{i,j\}\in E$) and
%\begin{equation}\label{matXG}
%\begin{array}{l}
%  \text{(i) }\  \ X_{\emptyset,\emptyset}=X_{\{i,j\},\{i,j\}}=1 \
%  \text{ for  all } \{i,j\}\in E,\\ 
%%\text{(ii) }\  X_{e,f}=X_{\emptyset,g} \text{ if } \{e,f,g\} \text{
%%is a triangle in }  G,\\ 
%\text{(ii) }\ X_{\{i,j\},\{j,k\}} = X_{\emptyset,\{i,k\}}  \text{ if } 
%(i,j,k) \text{ is a circuit (triangle) in } G,\\
%%\text{(iii) } X_{e,f}=X_{g,h} \text{ if } \{e,f,g,h\} \text{ is a
%%circuit in } G. 
%\text{(iii) } X_{\{i,j\},\{h,k\}}=X_{\{i,k\},\{j,h\}} \ \text{ if }  
%(i,j,h,k) \text{ is a circuit in } G.
%\end{array}
%\end{equation}

  \begin{example} \label{excircuit}
    % Here are some examples of graphs whose cut ideal is (or not)
    % $\TH_1$-exact.  Here are several examples of graphs cut-perfect
    % and non-cut-perfect graphs.
    %  \begin{enumerate} \label{ex:cutperfect}
    If $G$ has no circuit of length 3 or 4, then $\TH_1(IG) =
    [-1,1]^E$, since the conditions (\ref{matXG})(ii)-(iii) are void.
    For instance, if $G$ is a forest, then $\TH_1(IG) =
    [-1,1]^E=\CUT(G)$ and thus $IG$ is $\TH_1$-exact.  On the other
    hand, if $G=C_n$ is a circuit of length $n\geq 5$, then
    $\textup{TH}_1(IG) = [-1,1]^E$ strictly contains the polytope
    $\textup{CUT}(C_n)$ (as $|E|=n$ and $\CUT(C_n)$ has only $2^{n-1}$
    vertices).  Thus $I C_n$ is not $\TH_1$-exact for $n\ge 5$.
\end{example}

%See Section~\ref{seccircuit} for more results on theta bodies of the cut ideal of circuits. 

\begin{example}\label{exK5}
  For $G=K_5$, $I K_5$ is not $\TH_1$-exact.  Indeed, the inequality
  $\sum_{e \in E_5} x_e + 2 \geq 0$ induces a facet of
  $\textup{CUT}(K_5)$ (cf. e.g. \cite[Chapter 28.2]{DL97}) and the
  linear form $\sum_{e \in E_5} x_e + 2$ takes three distinct values
  on the vertices of $\textup{CUT}(K_5)$ (namely, $0$ on the facet,
  $12$ on the trivial empty cut and $4$ on the cut obtained by
  separating a vertex from all the others). Applying Theorem
  \ref{theo2level}, we can conclude that $I K_5$ is not $\TH_1$-exact.
  \end{example}

  In Section~\ref{seccompare} below we will characterize the graphs
  whose cut ideals are $\TH_1$-exact and we will determine the precise
  order $k$ at which the cut ideal of a circuit is $\TH_k$-exact in
  Section~\ref{seccircuit}.

\subsection{Comparison with other SDP relaxations of the cut
  polytope}\label{seccompare} 

% We mention here the link between the theta bodies of the cut ideal
% $IG$ and the semidefinite relaxations of the cut polytope introduced
% in \cite{LaurentMaxCut}.

% Let $G=([n],E)$ be a graph. 
% For $t\in\NN$, the following relaxation\footnote{For simplicity in the
%   notation we shift the indices by 1 with respect to
%   \cite{LaurentMaxCut}.}  $Q_t(G)$ of 
%$\textup{CUT}(G)$ is considered in \cite{LaurentMaxCut}.  

We mention here the link between the theta bodies of the cut ideal
$IG$ and some other semidefinite relaxations of the cut polytope
$\CUT(G)$. First note that the relaxation $\TH_1(IG)$ coincides with
the edge-relaxation considered by Rendl and Wiegele (see \cite{W}) and
numerical experiments there indicates that it is often tighter than
the basic semidefinite relaxation (\ref{sdpGW}) of $\CUT(G)$.

Next we compare the theta bodies of $IG$ with the relaxations $Q_t(G)$
of $\CUT(G)$ considered in \cite{LaurentMaxCut}\footnote{For
  simplicity in the notation we shift the indices by 1 with respect to
  \cite{LaurentMaxCut}.}. For $t \in \NN$, set $\OO_t(n):=\{T\subseteq
[n]\mid |T|\le t \text{ and } |T|\equiv t\mod 2\}$.  Then $Q_t(G)$
consists of the vectors $\yy \in \RR^{E}$ for which there exists a
positive semidefinite matrix $X$ indexed by $\OO_t(n)$ satisfying
(\ref{relTk}), $X_{T,T}=1$ ($T\in \OO_t(n)$), and
$y_{\{i,j\}}=X_{\emptyset,\{i,j\}}$ for $t$ even (resp., $y_{\{i,j\}}=
X_{\{i\},\{j\}}$ for $t$ odd) for all edges $\{i,j\}\in E$.
Therefore, for $t=1$, $Q_1(G)$ coincides with the Goemans-Williamson
SDP relaxation (\ref{sdpGW}).  Moreover, for even $t=2k$,
$Q_{2k}(K_n)$ coincides with the theta body $\TH_k(IK_n)$.
%, i.e. the two SDP relaxations coincide for complete graphs.
(To see it use Remark~\ref{remcomplete}.)
%For a graph $G=([n],E)$, let $Q_t(G):= \pi_E(Q_t(K_n))$ denote the
%projection of $Q_t(K_n)$ onto the edge set of $G$, which is the SDP
%relaxation of $\CUT(G)$ considered in \cite{LaurentMaxCut}.  
%Thus we have the following chain of inclusions:
The following chain of inclusions shows the link to the theta bodies:
\begin{equation}\label{inc}
 \CUT(G)\subseteq Q_{2k}(G) = \pi_E(Q_{2k}(K_n)) = \pi_E(\textup{TH}_k(IK_n))
  \subseteq \textup{TH}_k(IG)
\end{equation}
(where the last inclusion follows using (\ref{incproj})).  Therefore,
the $k$-th theta body $\TH_k(IG)$ is in general a weaker relaxation
than $Q_{2k}(G)$.  For instance, for the 5-circuit,
$\CUT(C_5)=Q_2(C_5)$ (see \cite{LaurentMaxCut}) but $\CUT(C_5)$ is
strictly contained in $\TH_1(IC_5) = [-1,1]^5$ (see
Example~\ref{excircuit}).

On the other hand, the SDP relaxation $\textup{TH}_k(IG)$ can be much
simpler and less costly to compute than $Q_{2k}(G)$, since its
definition exploits the structure of $G$ and thus often uses smaller
matrices.  Indeed, $Q_{2k}(G)$ is defined as the projection of
$Q_{2k}(K_n)$, whose definition involves matrices indexed by all even
sets $T \subseteq [n]$ of size at most $2k$, thus not depending on the
structure of $G$. On the other hand, the matrices needed to define
$\TH_k(IG)$ are indexed by the even sets $T \subseteq [n]$ of size at
most $2k$ for which $G$ has a $T$-join of size at most $k$.  
%%%%%%%
This can be checked efficiently since the minimum weight $T$-join problem has a polynomial time algorithm (cf. \cite{EJ73}).
For
instance, for $k=1$, $\textup{TH}_1(IG)$ uses matrices of size
$1+|E|$, while $Q_2(G)$ needs matrices of size $1+\binom{n}{2}$.
%%%%%%%%
We refer to Section \ref{seccomplexity} for more details on the complexity of constructing the above theta bodies.

\begin{example} \label{ex:Kn} It was shown in \cite{Lau03b}
  that $\CUT(K_n)$ is strictly contained in $Q_k(K_n)$ for $k < \lceil
  \frac{n}{2} \rceil -1$. Therefore, $\CUT(K_n) \subset \TH_k(IK_n) =
  Q_{2k}(K_n)$ for all $2k < \lceil \frac{n}{2} \rceil -1$. This
  implies that $IK_n$ is not $\TH_k$-exact for $k \leq \lfloor
  \frac{n-1}{4} \rfloor$. However, it is known that $\CUT(K_n) =
  Q_{\lceil \frac{n}{2} \rceil}(K_n)$ when $n \leq 7$. Therefore,
  $IK_5, IK_6$ and $IK_7$ are all $\TH_2$-exact.
\end{example}

For some graphs there is a special inclusion relationship between the
theta bodies and the $Q_t$-hierarchy. We consider first graphs with
bounded diameter.

%First we observe that $\TH_k(IG)$ refines $Q_2(G)$ for bounded
%diameter graphs. 

%For graphs with bounded diameter, there is a special inclusion
%relationship between the theta bodies and the $Q_t$-hierarchy. 

\begin{lemma}
\label{lemdiameter}
Let $G$ be a graph with diameter at most $k$, i.e., such that any two
vertices can be joined by a path traversing at most $k$ edges. Then
$\TH_k(IG)\subseteq Q_2(G)$.
\end{lemma}

\begin{proof}
  It suffices to observe that the set $\T_k$ indexing the matrices in
  the definition of $\TH_k(IG)$ (which consists of the even sets
  $T\subseteq V$ for which there is a $T$-join of size at most $k$)
  contains all pairs of vertices. Thus $\T_k$ contains the set
  $\OO_2(n)$ indexing the matrices in the definition of $Q_2(G)$.
\end{proof}

Next we observe that $\TH_k(IG)$ refines the Goemans-Williamson
relaxation~(\ref{sdpGW}) for graphs with radius $k$. 

\begin{lemma}\label{lemradius}
  Let $G$ be a graph with radius at most $k$, i.e., there exists a
  vertex that can be joined to any other vertex by a path traversing at
  most $k$ edges. Then $\TH_k(IG)\subseteq Q_1(G)$.
\end{lemma}

\begin{proof}
  Say vertex 1 can be joined to all other vertices $i\in [n]\setminus
  \{1\}$ by a path of length at most $k$. Then the set $\T_k$ contains
  $\emptyset$, $\{i,j\}$ for all edges $ij\in E$, and all pairs
  $\{1,i\}$ for $i\in [n]\setminus \{1\}$.  Let $\yy\in \TH_k(IG)$,
  i.e. there exists a positive semidefinite matrix $X$ indexed by
  $\T_k$ satisfying (\ref{relTk}) and $y_e=X_{\emptyset,e}$ for $e\in
  E$.  Consider the $n\times n$ matrix $Y$ defined by $Y_{ii}=1$
  ($i\in [n]$), $Y_{1i}=X_{\emptyset,\{1,i\}}$ ($i\in [n]\setminus
  \{1\}$), and $Y_{ij}=X_{\{1,i\},\{1,j\}}$ ($i\ne j\in [n]\setminus
  \{1\}$).  Then $Y\succeq 0$ (since $Y$ coincides with the principal
  submatrix of $X$ indexed by $\emptyset,\{1,2\},\ldots,\{1,n\}$),
  $y_{\{i,j\}}=Y_{ij}$ for all $\{i,j\}\in E$ (using
  (\ref{relTk})). This shows $\yy\in Q_1(G)$, concluding the proof.
\end{proof}

In particular, as already noted in \cite{W}, $\TH_1(IG)\subseteq
Q_1(G)$ if $G$ contains a vertex adjacent to all other vertices.  For
an arbitrary graph $G$, let $G^*$ be the graph obtained by adding
edges to $G$ so that one of its vertices is adjacent to all other
vertices. Thus, $\TH_1(IG^*)\subseteq Q_1(G^*)$ by Lemma
\ref{lemradius}. Taking projections onto the edge set of $G$, the
relaxation $\pi_E(\TH_1(IG^*))$ is contained in $\pi_E(Q_1(G^*)) =
Q_1(G)$ (and in $\TH_1(IG)$).

\subsection{Application to circuits in graphs}\label{secapplicircuit}

Let us consider briefly the case when $\M=\M_G$ is the graphic matroid
of a graph $G=(V,E)$, i.e. $\C=\C_G$ is the collection of cycles of
$G$ and $\C^*=\D_G$ is its collection of cuts.

One can find a set $\F$ of representatives for the
equivalence classes of (\ref{equirel}) as follows.  Namely, assume for
simplicity that $G$ is connected and let $E_0\subseteq E$ be the edge
set of a spanning tree in $G$.  Then the collection
$\F:=\PPP(E\setminus E_0)$ is a set of distinct representatives for
the classes of (\ref{equirel}).
% Indeed each edge $e\in E\setminus E_0$ closes a unique circuit $C_e$
% of $G$, and the collection $\{C_e\mid e\in E\setminus E_0\}$ forms a
% basis of the $\GF(2)$-cycle space $\C_G$.  As in the proof of
% Theorem \ref{theobaseM}, let $M$ be the matrix whose rows are the
% 0/1-characteristic vectors of the $C_e$'s. Then $M\one^F$ takes
% distinct values for the distinct subsets $F\subseteq E\setminus
% E_0$, which thus shows that those $F$'s are distinct representatives
% for $\sim$.
Indeed, note first that no two distinct subsets $F,F'$ of $E\setminus
E_0$ are in relation by $\sim$, since each non-empty cut meets the
tree $E_0$. Next, any subset $X\subseteq E_0$ determines a unique cut
$D_X$ for which $D_X\cap E_0=X$, so that $X\sim X\Delta D_X$.  Hence,
for any set $Z\subseteq E$, write $Z=X\cup Y$ with $X\subseteq E_0$
and $Y\subseteq E\setminus E_0$; then $Z\sim X\Delta D_X\Delta Y$ is
thus in the same equivalence class as a subset of $E\setminus E_0$.

Note however that the above set $\F$ may not consist of the minimum
cardinality representatives.  In fact, as we observe in the next section,
%as observed in Remark \ref{remrep}, 
finding a minimum cardinality representative in each
equivalence class amounts to solving a maximum weight cut problem,
thus a hard problem.  Nevertheless this collection $\F$ can be used to
index truncated moment matrices (simply index the $k$-th order matrix
by all $F\in \F$ with $|F|\le k$).  However, studying this SDP
hierarchy is less relevant for optimization purposes since the linear
inequality description of $\CYC(\M_G)$ is completely known (see
Theorem \ref{thm:classification} below), and one can find a maximum
weight cycle in a graph in polynomial time (cf. \cite{EJ73}).

\subsection{Computational complexity of theta bodies}\label{seccomplexity}

We group here some observations about the computational complexity of building the matrices 
$M_{\B_k}(\yy)$ needed to define  the theta body $\TH_k(I\M)$ of a binary matroid $\M$.

\medskip
The computation of the matrix $M_{\B_k}(\yy)$ is done in  two steps:\\
- First compute the sets $\F_k$ and $\F_{2k}$.\\
- Then compute the entries of $M_{\B_k}(\yy)$, i.e., for all $F,F'\in \F_k$, find $F''\in \F_{2k}$ such that 
$F\Delta F' \sim F''$,  i.e., $F\Delta F'\Delta F''\in\C^*$.

As we see below both steps involve  making repeated calls to a membership oracle for $\C^*$. Such a membership oracle gives a yes/no answer to any query of the form: 
\begin{center}Given $X\subseteq E$, does $X$ belong to $\C^*$?\end{center}
If the binary matroid $\M$ is given by a representation matrix $M$, then it is easy to answer such a query. Indeed, first bring $M$ into the form
$M= \left(\begin{matrix} I_r\  A\end{matrix}\right)$, where $A\in \{0,1\}^{r\times s}$
(using Gaussian elimination). Then  $M^*:= \left(\begin{matrix} A^T\ I_s\end{matrix}\right)$
is a representing matrix for $\M^*$ and thus  $X\in \C^*$ if and only if  $M^* \one^X=0$ modulo 2.

In the case  when $\M$ is the cographic matroid of a graph $G=(V,E)$, checking membership in $\C^*$ is easy:  $X\in \C^*$ (i.e., $X$ is a cycle of $G$)  if and only if each node $v\in V$ is adjacent to an even number of edges of $X$.

\medskip
 We first indicate how to construct the set $\F_k$. Say 
$\M = (E,\C)$ is a binary matroid on $E$, $m:=|E|$, and  $\PPP_k(E)$ denotes the collection of all subsets of $E$ with cardinality at most $k$.
As in (\ref{setF}), $\F$ denotes a set of distinct representatives for the 
 equivalence classes of the relation $\sim$ in (\ref{equirel}), and $\F_k$ is (as in (\ref{relFk})) a set of distinct representatives for those classes that contain at least one set of cardinality at most $k$. For any $F \subseteq E$, check that the subsets of $E$ that are equivalent to $F$ are precisely those of the form $F \Delta D$, $D \in \C^*$, and that for two distinct elements $D_1$ and $D_2$ in $\C^*$, $F \Delta D_1 \neq F \Delta D_2$. Therefore, $|\F|= 2^{m}/|\C^*|$ and $$|\F_k|\le |\PPP_k(E)|=\sum_{i=0}^k {m\choose i} =O(m^k).$$
 Such a set $\F_k$ can be constructed using the following simple procedure:
Order the elements of $E$ as $e_1,\ldots,e_m$ and order   the elements of $\PPP_k(E)$  by increasing cardinalities 
%according to a graded lexicographic order, i.e., as  
as $\emptyset$, $\{e_1\},\ldots,\{e_m\}$, 
$\{e_1,e_2\},\ldots,$ $\{e_{m-1},e_m\}$, $\ldots,$ $ \{e_1,\ldots,e_k\},\ldots,\{e_{m-k+1},\ldots,e_m\}$, denoted as $X_1=\emptyset,X_2,X_3, \ldots.$
We successively scan  the elements $X_1,X_2,\ldots$ and decide which ones should be selected in $\F_k$ in the following way: 
First select  $F_1:=X_1$. Then, say $X_1,\ldots,X_i$ have been scanned and we have selected $\{F_1,\ldots,F_r\}\subseteq \{X_1,\ldots,X_i\}$; if $X_{i+1}\Delta F_s\not\in \C^*$  $\ \forall s=1,\ldots,r$, then $F_{r+1}:=X_{i+1}$ is selected; otherwise we do not select it; go on with next set $X_{i+2}$.

Let $F_1,\ldots,F_p$ be the sets which have been selected by this procedure. Then, $F_r \not\sim F_{r'}$ for $1\le r<r'\le p$, and any $X\in \PPP_k(E)$ is equal or equivalent to one of $F_1,\ldots,F_p$. Therefore, the set $\{F_1,\ldots,F_p\}$ constitutes a set of distinct representatives for the classes of $\sim$ containing some set of size at most $k$ and thus it can be chosen for $\F_k$. Moreover it 
 has the property that any member of $\F_k$  has the smallest cardinality in its equivalence class.

The above  procedure makes a number of calls to a membership oracle in $\C^*$ which is of order $O(m^k)$, thus  polynomial in $m=|E|$, when $k$ is fixed.

\medskip
 Next we see how to construct the entries of  $M_{\B_k}(\yy)$.
For this we need to build  the multiplication table: for any $F,F'\in\F_k$, we must find the element $F''\in \F$ for which $\xx^F\xx^{F'}= \xx^{F''}$ modulo $I\M$ or, equivalently,
$F\Delta F'\sim F''$. As $|F\Delta F'|\le 2k$, $F''\in \F_{2k}$. Therefore, in order to build $M_{\B_k}(\yy)$ it suffices to build the set $\F_{2k}$ which, for fixed $k$,   can be done with a polynomial number of calls to a membership oracle for $\C^*$.

\medskip
\begin{remark}
On the other hand,  let us point out that, given $F\subseteq E$, the problem of finding a representative  $F'$ of minimum cardinality in the equivalence class of $F$ is hard in general. (Of course, if we fix the cardinality of $F$ the problem becomes easy as we just observed.)
 Indeed this is the problem:
\begin{equation}\label{pb1}
\min \ |D\Delta F| \ \text{ such that } D\in \C^*
\end{equation}
or equivalently 
\begin{equation}\label{pb2}
 \max\  \ww^T\xx \ \text{ such that } \xx\in {\rm CYC}(\M^*), 
 \end{equation}
after defining $\ww\in \oR^E$ by $w_e=-1$  for $e\in F$ and $w_e =1$ for $e\in E\setminus F$
(and noting that $\ww^T\chi^D=|E|-2|F\Delta D|$).

When $\M$ is a cographic matroid, the above problem (\ref{pb1}) asks to find a minimum cardinality $T$-join (where $T$ is the set of odd degree nodes in $F$) which can be solved in polynomial time (see \cite{EJ73}).

When $\M$ is a graphic matroid, the above problem (\ref{pb2}) is an instance of the maximum cut problem, which  is an NP-hard problem for general graphs. However (\ref{pb2})  is polynomial time solvable if $\M$ has no $\M(K_5)$-minor (see \cite{Bar83}).

The problem (\ref{pb2}) is also polynomial time solvable when $\M$ or $\M^*$ does not have $F_7$ or $\M(K_5)^*$ as a minor; see \cite{GT89} for details and for other classes of matroids for which (\ref{pb2}) is polynomial time solvable. 
\end{remark}

\section{Matroids whose cycle ideals are
  $\TH_1$-exact} \label{secexact} 

\subsection{Matroid minors} \label{secmatroidminors}
Let $\M=(E,\C)$ be a binary matroid and $e\in E$. Set
$$\C \backslash e:=\{C\in \C\mid e\not\in C\}, \qquad
\C/e:=\{C\setminus\{e\}\mid C\in \C\}.$$ 
Then, $\M\backslash e:=
(E\setminus\{e\},\C\backslash e)$ and $\M/e:=(E\setminus\{e\},\C/e)$
are again binary matroids; one says that $\M\backslash e$ is obtained
by {\em deleting} $e$ and $\M/e$ by {\em contracting} $e$. A {\em
  minor} of $\M$ is obtained by a sequence of deletions and
contractions, thus of the form $\M\backslash X / Y$ for disjoint
$X,Y\subseteq E$.  In the language of binary spaces, $C\backslash e$
arises from $\C$ by taking the intersection with the hyperplane
$x_e=0$, while $\C/e$ arises by projecting $\C$ onto $\RR^{E\setminus
  \{e\}}$.

\begin{example} \label{exPr} Let $M_r$ denote the $r\times (2^r-1)$
  matrix whose columns are all non-zero 0/1 vectors of length $r$, and
  let $\PPP_r$ denote the binary matroid represented by $M_r$, called
  the binary projective space of dimension $r-1$. One can verify that
  $\PPP_r$ has $2^r$ cocycles; the non-empty cocycles have size
  $2^{r-1}$ and thus are cocircuits.  Hence, $\CYC(\PPP_r^*)$ is a
  simplex and $I\PPP_r^*$ is $\TH_1$-exact. When $n=3$, $\PPP_3 =:
  F_7$ is called the {\em Fano matroid}. It will follow from
  Theorem~\ref{theochar} that $IF_7$ is also $\TH_1$-exact.

%   $F_7$ is the binary matroid on $E=\{1,\ldots,7\}$ represented by the
%   matrix
% $$\bordermatrix{& 1 & 2 & 3 & 4 & 5 & 6 & 7\cr
% &1 & 1 & 0 & 1 & 0 & 0 & 0\cr
% &1 & 0 & 1 & 0 & 1 & 0 & 0\cr
% &0 & 1 & 1 & 0 & 0 & 1 & 0\cr
% &1 & 1 & 1 & 0 & 0 & 0 & 1},$$
% with circuits
% $124, 135, 167, 236, 257, 347, 456$ (the lines of the Fano plane) 
% and their complements. Thus
% the cycles of $F_7$ are 
% $\emptyset$, $C$, $E\setminus C$, $E$ (for $C$ circuit of $F_7$), and
% the non-empty cocycles are the complements of the lines of the Fano
% plane.  Therefore, $\CYC(F_7^*)$ (the cycle polytope of $F_7^*$) is a
% simplex and thus it is $\TH_1$-exact.  Moreover, $\CYC(F_7)$ too is
%$\TH_1$-exact (this follows from Theorem \ref{theochar} below).
\end{example}

\begin{example}\label{exR10}
$R_{10}$ is the  binary matroid on 10 elements, represented by the matrix
$$\bordermatrix{ & 34 & 35& 45& 23 & 24 & 25 & 13 & 14 & 15 & 12\cr
&1 & 1 & 1 & 1 & 1 & 1 & 0 & 0 & 0 & 0 \cr
&1 & 1 & 1 & 0 & 0 & 0 & 1 & 1 & 1 & 0 \cr
&1 & 0 & 0 & 1 & 1 & 0 & 1 & 1 & 0 & 1 \cr
&0 & 1 & 0 & 1 & 0 & 1 & 1 & 0 & 1 & 1 \cr
&0 & 0 & 1 & 0 & 1 & 1 & 0 & 1 & 1 & 1
},$$
where it is convenient to index the columns by the edge set $E_5$ of
$K_5$.  Then the cycles of $R_{10}$ correspond to the even cycles of
$K_5$, and the cocycles of $R_{10}$ to the cuts of $K_5$ and their
complements.  Note that $R_{10}$ is isomorphic to its dual.  Consider
the inequality:
\begin{equation}\label{ineqR10}
  \sum_{e\in F}x_e-\sum_{e\in E_5\setminus F}x_e  \ge -4,
\end{equation}
where $F$ consists of three edges adjacent to a common vertex (e.g.
$F=\{12,13,14\}$). (Thus (\ref{ineqR10}) is of the form (\ref{ineqmet}),
but with a shifted right hand side.) One can verify that
(\ref{ineqR10}) defines a facet of $\CYC(R_{10})$ and that the linear
function in (\ref{ineqR10}) takes three distinct values on the cycles
of $R_{10}$ (namely, 0, 4, and -4).  Therefore, in view of Theorem
\ref{theo2level}, we can conclude that $R_{10}$ is not $\TH_1$-exact.
\end{example}

\subsection{The cycle polytope} \label{cyclepolytope}

As each cycle and cocycle have an even intersection, the following
inequalities are valid for the cycle polytope $\CYC(\M)$:
\begin{equation}\label{ineqmet}
\sum_{e\in F}x_e-\sum_{e\in D\setminus F}x_e \ge 2-|D| \ \text{ for } 
D\in \C^*, F\subseteq D, |F| \text{ odd.}
\end{equation}
Let $\MET(\M)$ be the polyhedron in $\RR^E$ defined by the
inequalities (\ref{ineqmet}) together with $-1\le x_e\le 1$ ($e\in
E$).  We have $\CYC(\M)\subseteq \MET(\M)$. In particular, $\CYC(\M)$
is contained in the hyperplane $x_e=1$ if $e$ is a coloop of $\M$, and
it is contained in the hyperplane $x_e-x_f=0$ if $e,f$ are
coparallel. Thus we may assume without loss of generality that $\M$
has no coloops and no coparallel elements.  We will use the following
known results.

\begin{lemma}\cite[Corollary~4.21]{Barahona-Groetschel} \label{lemfacet}
  Let $\M$ be a binary matroid with no $F_7^*$ minor.  The inequality
  (\ref{ineqmet}) defines a facet of $\CYC(\M)$ if and only if $D$ is
  a chordless cocircuit of $\M$.
\end{lemma}

% \begin{lemma}\cite[Lemma~4.17]{Barahona-Groetschel} \label{lem:intersection} 
%Let  $D$ be a cocircuit of $\M$. For each even subset $F$ of $D$
%  there exists a cycle $C$ of $\M$ such that $D \cap C = F$.
%\end{lemma}

\begin{theorem}
  \cite[Theorem~4.22]{Barahona-Groetschel} \label{thm:classification}
  For a binary matroid $\M$, $\CYC(\M)=\MET(\M)$
%without coloops or coparallel elements 
if and only if $\M$ has no $F_7^*$, $R_{10}$ or $\M_{K_5}^*$ minors.
\end{theorem}

Recall that $I\M$ is $\TH_1$-exact if $\CYC(\M) = \TH_1(I\M)$.

\begin{lemma}\label{lem1}
  Assume $\M$ has no $F_7^*$ minor. If $I\M$ is $\TH_1$-exact then
  $\M$ does not have any chordless cocircuit of length at least five.
\end{lemma}

\begin{proof}
  Suppose $D=\{e_1,\ldots,e_k\}$ is a chordless cocircuit of $\M$ with
  $k=|D|\ge 5$.
% and let $e\in D$.  
By Lemma~\ref{lemfacet}, the inequality 
\[
x_{e_1}-x_{e_2}-\dots -x_{e_k} \ge 2-k
\] 
defines a facet of $\CYC(\M)$.  We now use the following claim
\cite[Lemma~4.17]{Barahona-Groetschel}: For each even subset
$F\subseteq D$, there exists a cycle $C\in \C$ for which $C\cap
D=F$. Thus we can find three cycles whose intersections with $D$ are
respectively $\emptyset$, $\{e_2,e_3\}$ and $\{e_2,e_3,e_4,e_5\}$.
Then the linear form $x_{e_1}-x_{e_2}-\dots -x_{e_k}$ evaluated at
each of these three cycles takes the values $2-k,6-k,10-k$. In view of
Theorem \ref{theo2level} we can thus conclude that $I\M$ is not
$\TH_1$-exact.
\end{proof}

\begin{theorem}\label{theochar}
  Assume $\M$ has no $F_7^*$, $R_{10}$ or $\M_{K_5}^*$ minors.  Then
  $I\M$ is $\TH_1$-exact if and only if $\M$ does not have any
  chordless cocircuit of length at least 5.
\end{theorem}

\begin{proof}
  Lemma \ref{lem1} gives the `only if' part. For the `if' part, it
  suffices to verify that, if $D$ is a cocircuit of length at most 4
  and $F$ is an odd subset of $D$, then the linear form $\sum_{e \in
    F} x_e - \sum_{e \in D \setminus F} x_e $ takes two values when
  evaluated at cycles of $\M$, and then to apply Theorems
  \ref{thm:classification} and \ref{theo2level}.
\end{proof}

\begin{corollary}\label{corgraphic}
  The cycle ideal of a graphic matroid $\M_G$ is $\TH_1$-exact if and
  only if $G$ has no chordless cut of size at least 5.
\end{corollary}

\begin{proof}
  Directly from Theorem \ref{theochar} since graphic matroids do not
  have $F_7^*$, $R_{10}$ or $\M_{K_5}^*$ minors.
\end{proof}

\begin{lemma}\label{lemdeletion}
  If $I\M$ is $\TH_k$-exact, then the cycle ideal of any deletion
  minor of $\M$ is also $\TH_k$-exact.
\end{lemma}

\begin{proof}
  Say $\M' = \M\backslash e_1$ is a deletion minor of $\M$, where
  $E=\{e_1,\ldots,e_m\}$ and $E'=E\setminus\{e_1\}$. Take $\xx'\in
  \TH_k(I\M')$; we show that $\xx'\in\CYC(\M')$. For this extend
  $\xx'$ to $\xx\in \RR^E$ by setting $x_{e_1}:=1$. We verify that
  $\xx \in \TH_k(I\M)$.

  For this consider a linear polynomial $f\in \RR E$ of the form
  $f=s+q$ where $s$ is a sos of degree at most $2k$ and $q\in I\M$.
  Define the polynomials $f',s',q'\in \RR E'$ by
  $f'(x_{e_2},\ldots,x_{e_m})=f(1,x_{e_2},\ldots,x_{e_m})$; similarly
  for $q',s'$.  Obviously $s'$ is sos with degree at most $2k$. Since
  $q$ vanishes on $\{ \chi^C \,:\, C \in \C \}$, it vanishes on all
  $\chi^C$, $C \in \C$, with $x_{e_1} = 1$. This last fact is
  equivalent to saying that $q'$ vanishes on $\{ \chi^C \,:\, C \in
  \C' \}$.
% Moreover, $q'\in I\M'$ since for a cocircuit $D$ of $\M$
%   containing $e_1$, the set $D\setminus \{e_1\}$ is a cocycle of $\M'$
%   (and using the description of the cycle ideal of Theorem
%   \ref{theobaseM}). 
Therefore, $f'$ is $k$-sos modulo $I\M'$ and so $f'(\xx')\ge 0$ as
$\xx'\in \TH_k(I\M')$. In particular, $f(\xx)=f'(\xx')\ge 0$ and 
$\xx\in \TH_k(I\M)=\CYC(\M)$.

  Thus $\xx$ is a convex combination of $\pm 1$-incidence vectors of
  cycles of $\M$; as $x_{e_1}=1$ no cycle in the combination uses
  $e_1$, which thus gives a decomposition of $\xx'$ as a convex
  combination of cycles of $\M'$.
\end{proof}

\begin{remark}
  On the other hand, the property of being $\TH_1$-exact is {\em not}
  preserved under taking contraction minors.  Indeed, every binary
  matroid can be realized as a contraction minor of some dual binary
  projective space $\PPP_r^*$ (see \cite{GT89b}). Now we observed in
  Example \ref{exPr} that the cycle ideal of $\PPP_r^*$ is
  $\TH_1$-exact, while $I\M$ is not always $\TH_1$-exact.
\end{remark}

See Section 5 for examples of cographic matroids whose cycle ideal is
$\TH_2$-exact while they have a contraction  minor whose cycle ideal is
not $\TH_k$-exact for large $k$ (this is the case for wheels,
cf. Corollary \ref{cordiameter}).

% We now arrive at a characterization of ``cut-perfect'' graphs
% answering Problem 8.4 in \cite{Lovasz}.

% \begin{corollary}\label{corcographic}
%   The cut ideal $IG$ of a graph $G$ is $\TH_1$-exact (i.e.  the cycle
%   ideal of the cographic matroid $\M_G^*$ is $\TH_1$-exact) if and
%   only if $M_G^*$ has no $M_{K_5}^*$ minor or, equivalently, $G$ has
%   no $K_5$ minor, and $G$ has no chordless circuit of size at least 5.
% \end{corollary}

% \begin{proof} 
%   % Assume $\M = \M_G^\ast$ is $\TH_1$-exact. If $\M_{K_5}^\ast$ is a
%   % minor of $\M$, then $\M_{K_5}^\ast = \M_G^\ast / T \backslash S =
%   % \M_{G \backslash T / S}^\ast$ for some disjoint $S,T\subseteq E$.
%   Assume $\M = \M_G^\ast$ is $\TH_1$-exact. As $M_{K_5}^*$ is not
%   $\TH_1$-exact (cf. Example \ref{exK5}), Lemma \ref{lemdeletion}
%   implies that $M_{K_5}^*$ is not a deletion minor of $\M$.  Therefore
%   $K_5$ is not a (graph) contraction minor of $G$. This also implies
%   that $K_5$ is not a minor of $G$ (as $K_5$ is a complete graph).
%   Next, this implies that $M_{K_5}^*$ is not a minor of $\M$.  Indeed,
%   if $M_{K_5}^*$ is a minor of $\M$ then, by Whitney's 2-isomorphism
%   theorem (cf. \cite{Oxley}), $K_5$ is 2-isomorphic to a minor $H$ of
%   $G$; but then $H$ must be isomorphic to $K_5$ as the the only graph
%   2-isomorphic to $K_5$ is $K_5$ itself.  As $\M$ contains no $F_7^*$
%   or $R_{10}$ minor, Theorem \ref{theochar} gives the reverse
%   implication: If $M_{K_5}^*$ is not a minor of $\M$, then $\M$ is
%   $\TH_1$-exact. Thus equivalence holds throughout and the proof is
%   complete.
% \end{proof}

We now characterize the $\TH_1$-exact cographic matroids.  We begin
with a lemma relating graph and matroid minors involving $K_5$.

\begin{lemma}\label{lemK5}
  The cographic matroid $\M_G^*$ of a graph $G$ has a $\M_{K_5}^*$
  minor if and only if $K_5$ is a contraction minor of $G$.
\end{lemma}

\begin{proof}
  The `if part' is obvious since if $K_5$ is a contraction minor of
  $G$, then $\M_{K_5}^*$ is a deletion minor of $\M_G^*$.  Conversely
  assume that $\M_{K_5}^*$ is a minor of $\M_G^*$.  By Whitney's
  2-isomorphism theorem (cf. \cite{Oxley}), $K_5$ is 2-isomorphic to a
  minor $H$ of $G$; but then $H$ must be isomorphic to $K_5$ as the
  the only graph 2-isomorphic to $K_5$ is $K_5$ itself.  Hence $K_5$
  is a minor of $G$, which implies that $K_5$ is also a contraction
  minor of $G$.
\end{proof}

\begin{corollary}\label{corcographic}
  The cycle ideal of a cographic matroid $\M_G^*$ is $\TH_1$-exact if
  and only if $M_G^*$ has no $M_{K_5}^*$ minor and no chordless
  cocircuit of length at least 5.
\end{corollary}

\begin{proof}
  Note that $\M_G^*$ contains no $F_7^*$ or $R_{10}$ minor. Hence in
  view of Theorem \ref{theochar}, it suffices to show that if $\M_G^*$
  is $\TH_1$-exact then $\M_G^*$ has no $M_{K_5}^*$ minor. So assume
  that $\M_G^*$ is $\TH_1$-exact.  As $\M_{K_5}^*$ is not
  $\TH_1$-exact (cf. Example \ref{exK5}), Lemma \ref{lemdeletion}
  implies that $\M_{K_5}^*$ is not a deletion minor of
  $\M_{G}^*$. Hence $K_5$ is not a contraction minor of $G$ which,
  by Lemma \ref{lemK5}, implies that $\M_{K_5}^*$ is not a minor of
  $\M_{G}^*$.
\end{proof}

Reformulating this last result we arrive at a characterization of
`cut-perfect' graphs, answering Problem 8.4 in [19].

\begin{corollary}\label{corcutperfect}
  The cut ideal of a graph $G$ is $\TH_1$-exact if and only if $G$ has
  no $K_5$ minor and no chordless circuit of length at least 5.
\end{corollary}

In \cite[Theorem~3.2]{Sullivant}, Sullivant obtains the same
characterization for {\em compressed} cut polytopes; namely he proves
that $\CUT(G)$ is compressed if and only if $G$ has no $K_5$ minor and
no chordless cycles of length at least 5. See \cite[Section 4]{GPT}
for comments on the connection between compressed polytopes and
$\TH_1$-exactness.

\section{The theta bodies for cut ideals of
  circuits}\label{seccircuit} 

% In this section we determine the exact order $k$ for which the cut
% ideal $IC_n$ of a circuit $C_n$ with $n$ edges is $TH_k$-exact,

% \begin{theorem}\label{theoC1}
% The equality $\TH_k(IC_n)=\CUT(C_n)$ holds if $n\le 4k$.
% \end{theorem}

In this section we determine the exact order $k$ for which the cut
ideal $IC_n$ of a circuit $C_n$ with $n$ edges is $\TH_k$-exact.  We
also obtain some results on graphs whose cut ideal is $\TH_2$-exact.
We begin with a result determining when the inequalities (\ref{ineqmet})
associated to circuits of $G$ are valid for $\TH_k(IG)$.

\begin{theorem}\label{theoC0}
  Let $C$ be a circuit of a graph $G$, let $e\in C$, and let $k$ be an
  integer such that $4k\ge |C|$. Then the inequality
\begin{equation}\label{eqC}
x_e -\sum_{f\in C\setminus \{e\}} x_f \ge 2-|C|
\end{equation}
is valid for $\TH_k(IG)$.
\end{theorem}

The proof uses the following preliminary results.  For convenience,
for a graph $G=(V,E)$, let $\SSS_{k}$ denote the set of polynomials
$f\in\RR E$ that are $k$-sos modulo the cut ideal $IG$.

\begin{lemma}\label{lemC1}
For a graph $G$, let 
 $F_1,F_2,F_3,F_4\subseteq E$ with $|F_i|\le k$ and 
such that $F_1\Delta F_2\Delta F_3\Delta F_4 $ is a cycle of $G$.
Then $2+\xx^{F_1}-\xx^{F_2}-\xx^{F_3}-\xx^{F_4}\in \SSS_{k}$.
\end{lemma}

\begin{proof}
  We use the following fact: As $C:=F_1\Delta F_2\Delta F_3\Delta F_4$
  is a cycle, $1-\xx^C\in IG$ by Theorem \ref{theobaseM}, and thus
  $1\equiv \xx^C \equiv \xx^{F_1}\xx^{F_2}\xx^{F_3}\xx^{F_4}$ modulo
  $IG$. This implies that $\xx^{F_i}\xx^{F_j} \equiv
  \xx^{F_k}\xx^{F_l}$ for $\{i,j,k,l\} = \{1,2,3,4\}$.  Now, one can
  easily verify that $(2+\xx^{F_1}-\xx^{F_2}-\xx^{F_3}-\xx^{F_4})^2
  \equiv 4(2+\xx^{F_1}-\xx^{F_2}-\xx^{F_3}-\xx^{F_4})$ modulo $IG$,
  which gives the result.
\end{proof}

\begin{lemma}\label{lemC2}
For a graph $G$, let $A,B\subseteq E$ with $|A|,|B|,|A\Delta B|\le k$.
Then 
 $1+\xx^A -\xx^B-\xx^{A\Delta B}\in \SSS_{k}$.
\end{lemma}

\begin{proof}
  We have $(1+\xx^A -\xx^B-\xx^{A\Delta B})^2 \equiv 4+2(\xx^A
  -\xx^B-\xx^{A\Delta B}) + 2( -\xx^A\xx^B-\xx^A\xx^{A\Delta
    B}+\xx^B\xx^{A\Delta B}) \equiv 4(1+\xx^A -\xx^B-\xx^{A\Delta B})$
  modulo $IG$.
\end{proof}

\begin{lemma}\label{lem3}
For a graph $G$, let $F\subseteq E$, $e\in F$, and $k\ge |F|$.
Then:
\begin{equation}\label{relC}
\begin{array}{l}
\text{ (i) } \ \ k-1 +x_e -\sum_{f\in F \setminus \{e\}}x_f -\xx^F \in
\SSS_{k},\\ 
\text{ (ii) } \ k-1 -\sum_{f\in F} x_f +\xx^F  \in \SSS_{k}.
\end{array}
\end{equation}
\end{lemma}

\begin{proof}
  It suffices to show the result for $k=|F|$. We show (i) using
  induction on $k\ge 2$. (The proof for (ii) is analogous.)  For
  $k=2$, $F=\{e,f\}$, we have $1+x_e - x_f -x_ex_f \in \SSS_{2}$ by
  Lemma~\ref{lemC2}. Consider now $|F|=k\ge 3$ and let $g\in
  F\setminus \{e\}$.  By the induction assumption applied to the set
  $F\setminus \{g\}$, we have:
$$k-2 +x_e-\sum_{f\in F\setminus\{e,g\}}x_f -\xx^{F\setminus \{g\}}
\in\SSS_{k-1}\subseteq \SSS_{k}.$$ 
Applying Lemma \ref{lemC2} to the sets $F\setminus\{g\}$, $\{g\}$ and
$F$, we obtain
$$1+ \xx^{F\setminus\{g\}} -x_g-\xx^F\in \SSS_{k}.$$
Summing up the above two relations yield the desired relation
(\ref{relC})(i).
%$k-1 + x_e -\sum_{f\in F\setminus \{e\}} -\xx^F \in  \SSS_{k}$.
\end{proof}

\begin{proof} ({\em of Theorem \ref{theoC0}}) Let $C$ be a circuit in
  $G$ with $|C|\le 4k$, i.e.  $k\ge m:=\lceil |C|/4\rceil$.  Let $F$
  denote the edge set of $C$ and let $e\in F$.  We show that the
  linear polynomial $f_{C}:= x_e-\sum_{f\in F\setminus \{e\}}x_f
  +|C|-2$ is $k$-sos modulo $IG$.  For this we consider a partition of
  $F$ into four sets $F_1,\ldots,F_4$ with $|F_i|\le m\le k$ for
  $i=1,\ldots,4$; say $e\in F_1$. Applying Lemma \ref{lemC1}, we
  obtain that
$$2+\xx^{F_1}-\xx^{F_2}-\xx^{F_3}-\xx^{F_4}\in \SSS_{k}.$$
Next, applying the condition (\ref{relC})(i) to $F_1$ we obtain
$$|F_1|-1 +x_{e} - \sum_{f\in F_1\setminus \{e\}}x_f -\xx^{F_1} \in \SSS_k,$$
and applying the condition (\ref{relC})(ii) to $F_i$ yields
$$|F_i|-1 -\sum_{f\in F_i}x_f +\xx^{F_i} \in \SSS_k \ \ \forall i=2,3,4.$$
Summing up the above relations yields the desired result, namely 
$f_{C}$ is $k$-sos modulo $IG$ and thus $f_{C} \geq 0$ is valid
for $\TH_k(IG)$.
\end{proof}

%\begin{proof} ({\em of Theorem \ref{theoC0}}) Consider the circuit
%  $C_n=([n],E)$ with $n$ edges and $n\le 4k$, i.e. $k\ge m:=\lceil
%  n/4\rceil$. We show that the linear polynomial $f_{C_n}:=
%  x_e-\sum_{f\in E\setminus \{e\}}x_f +n-2$ is $k$-sos modulo $IC_n$.
%  For this we consider a partition of the edge set of $C_n$ into four
%  sets $F_1,\ldots,F_4$ with $|F_i|\le m\le k$ for $i=1,\ldots,4$. Say
%  $e\in F_1$. As $F_1,\ldots,F_4$ partition $C_n$ we obtain from Lemma
%  \ref{lemC1} that
%$$2+\xx^{F_1}-\xx^{F_2}-\xx^{F_3}-\xx^{F_4}\in \SSS_{k}.$$
%Next, applying the condition (\ref{relC})(i) to $F_1$ we obtain
%$$|F_1|-1 +x_{e} - \sum_{f\in F_1\setminus \{e\}}x_f -\xx^{F_1} \in \SSS_k,$$
%and applying the condition (\ref{relC})(ii) to $F_i$ yields
%$$|F_i|-1 -\sum_{f\in F_i}x_f +\xx^{F_i} \in \SSS_k \ \ \forall i=2,3,4.$$
%Summing up the above relations yields the desired result
%$f_{C_n}$ is $k$-sos modulo $IC_n$ and hence $f_{C_n} \geq 0$ is valid
%for $\TH_k(IC_n)$.
% 
%Now suppose $C_n$ is a circuit of $G$. Then since every cut in $C_n$
%is the restriction of a cut in $G$, $IC_n$ is the elimination ideal of
%$IG$ with respect to the edge set $E$ of $C_n$. Therefore, by
%(\ref{incproj}), $\pi_E(\TH_k(IG)) \subseteq \TH_k(IC_n)$, and if
%$f_{C_n} \geq 0$ is valid for $\TH_k(IC_n)$ then it is also valid for
%$\TH_k(IG)$.
%\end{proof}

\begin{corollary}\label{corC1}
  For the circuit $C_n$ of length $n$, the equality
  $\TH_k(IC_n)=\CUT(C_n)$ holds for $n\le 4k$.
\end{corollary}

\begin{proof}
  Consider the circuit $C_n=([n],E)$ with $n\le 4k$. By Theorem
  \ref{thm:classification}, the complete linear description of
  $\CUT(C_n)$ is provided by the inequalities (i) $\sum_{e\in
    F}x_e-\sum_{e\in E\setminus F}x_e \ge 2-n$ where $F$ is any odd
  subset of $E$, and (ii) $-1 \leq x_e \leq 1$ for all $e \in E$.
  Thus in order to show $\TH_k(IC_n)=\CUT(C_n)$, it suffices to show
  that the inequalities (i),(ii) are all valid for $\TH_k(IC_n)$. This
  is obvious for (ii).  Using the well-known switching symmetries of
  the cut polytope (cf. \cite{Barahona-Groetschel}, \cite{DL97}), it
  suffices to show the desired property for the inequalities (i) with
  $|F|=1$.  But this result has just been shown in
  Theorem~\ref{theoC0}.
\end{proof}

\begin{lemma}\label{lemC4}
%We have strict inclusion $\CUT(C_n)\subset \TH_k(IC_n)$ if $n\ge 4K=1$.
If $n\ge 4k+1$, then $\TH_k(IC_n)=[-1,1]^E$.
\end{lemma}

\begin{proof} In view of Remark \ref{remconstraint}, it suffices to
  observe that the constraints (\ref{matX}) defining the theta body
  $\TH_k(IC_n)$ reduce to the constraints (\ref{matX})(i) and
  (\ref{matXmom}). Let $\F_k$ be the set indexing the combinatorial
  moment matrices in the definition of $\TH_k(IC_n)$, where we can
  assume that each $F_i \in \F_k$ has cardinality at most $k$. Now
  consider a constraint of type (\ref{matX})(ii). Since
  $F_1,\ldots,F_4\in\F_k$ have size at most $k$ and $\Delta_i F_i$ is
  a cycle of $C_n$, this cycle must be the empty set since $|\Delta_i
  F_i| \leq 4k < n$. Therefore we have a constraint of type
  (\ref{matXmom}).
\end{proof}

\begin{corollary}\label{corordercircuit}
The smallest order $k$ at which $IC_n$ is $\TH_k$-exact is $k=\lceil
n/4\rceil$. 
\end{corollary}

\begin{proof}
Directly from Theorem \ref{theoC0} and Lemma \ref{lemC4}.
\end{proof}

\begin{remark}\label{remcircuit}
  One can verify that the linear form $x_e-\sum_{f\in C_n\setminus
    \{e\}}x_f$ takes $\lfloor (n+1)/2\rfloor$ distinct values at the
  cut vectors of the circuit $C_n$. By (\ref{relklevel}), this permits
  to conclude that $IC_n$ is $\TH_k$-exact for $k= \lfloor
  (n+1)/2\rfloor-1$. This value is however larger than the order
  $\lceil n/4\rceil$ shown in Corollary \ref{corordercircuit} (for
  $n\ge 6$).  Thus the reverse implication of (\ref{relklevel}) does
  not hold.
\end{remark}

\begin{corollary}\label{cornoC9}
If the graph $G$ has no $K_5$ minor and no chordless circuit of length at
least 9, then its cut ideal $IG$ is $\TH_2$-exact.
\end{corollary}

\begin{proof}
Direct application of Theorems~\ref{thm:classification} and~\ref{theoC0}.
\end{proof}

Note that the reverse implication in Corollary \ref{cornoC9} does not
hold. We will see below (in Corollary \ref{cordiameter}) that the cut
ideal of a wheel is $\TH_2$-exact, but a wheel can contain a chordless
circuit of arbitrary length.

While we could characterize the graphs whose cut ideal is $\TH_1$-exact,
it is an open problem to characterize the graphs whose cut ideal is
$\TH_2$-exact. We
conclude this section with several observations about these graphs.

\begin{corollary}\label{cordiameter}
If the graph $G$ has no $K_5$ minor and has diameter at most~2 then its
cut ideal
$IG$ is $\TH_2$-exact.
\end{corollary}

\begin{proof}
  As $G$ has diameter at most 2, Lemma \ref{lemdiameter} gives the
  inclusion $\TH_2(IG)\subseteq Q_2(G)$. It was shown in
  \cite{LaurentMaxCut} that if $G$ has no $K_5$ minor then
  $Q_2(G)=\CUT(G)$.
\end{proof}

A {\em wheel} of length $n$ is a graph consisting of a circuit of
length $n$ with an additional vertex adjacent to all vertices on the
circuit. As wheels have no $K_5$ minor and their diameter is 2, their
cut ideal is $\TH_2$-exact.  Hence, within graphs with no $K_5$
minors, the cut ideal is $\TH_2$-exact for the following two classes:
graphs with diameter at most~2 and graphs with no chordless circuit of
size at least~9. Note that there is no containment between these two
classes; e.g. wheels of length $n\ge 9$ have diameter 2 but contain a
circuit of length $n$, and $C_8$ has diameter larger than 2.

The following further graphs have a $\TH_2$-exact cut ideal:
$K_5,K_6,K_7$ (and probably $K_8$ too, as conjectured in
\cite{LaurentMaxCut}). Finally, if the cut ideal of a graph $G$ is
$\TH_2$-exact, then the same holds for the cut ideal of any {\em
  contraction} minor $H$ of $G$; in particular, $C_9$ is not a
contraction minor of $G$.


\begin{thebibliography}{99}

\bibitem{Bar83}
F. Barahona.
The max-cut problem on graphs not contractible to $K_5$.
{\em Operations Research Letters}, {\bf 2}:107--111, 1983.


\bibitem{Barahona-Groetschel}
F. Barahona and M. Gr\"otschel.
On the cycle polytope of a binary matroid.
{\em Journal of Combinatorial Theory. Series B}, 40(1):40--62, 
1986.

\bibitem{BarahonaMahjoub86}
F. Barahona and A.-R. Mahjoub.
\newblock On the cut polytope.
\newblock {\em Math. Programming}, 36(2):157--173, 1986.



\bibitem{BCR}
J. Bochnak, M. Coste and M.-F. Roy.
{\em Real Algebraic Geometry},
Springer, 1998.

\bibitem{CLO}
D.A. Cox, J.B. Little and D.B. O'Shea.
{\em Ideals, Varieties and Algorithms: An Introduction to
  Computational Algebraic Geometry and Commutative Algebra}, 
Springer, 2005.

\bibitem{DL97}
M.M. Deza and M. Laurent.
{\em Geometry of Cuts and Metrics}.
Springer, 1997.

\bibitem{EJ73}
J. Edmonds and E.L. Johnson.
Matching, Euler tours and the Chinese postman.
{\em Mathematical Programming} {\bf 5}:88--124, 1973.

%\bibitem{Ed65} J. Edmonds.  The Chinese postman's problem.  {\em
 %   Bulletin of the Operations Research Society of America} {\bf 13    B-73}, 1965.

\bibitem{GJ79} M.R. Garey and D.S. Johnson.  {\em Computers and
    Intractability: A Guide to the Theory of NP-Completeness}, San
  Francisco, W.H. Freeman \& Company, Publishers, 1979.

\bibitem{GW95}
M.X. Goemans and D. Williamson.
Improved approximation algorithms for maximum cuts
and satisfiability problems using semidefinite programming.
{\em Journal of the ACM} {\bf 42}:1115--1145, 1995.

\bibitem{GPT} 
J.  Gouveia, P.A. Parrilo, and R. Thomas.  \newblock
Theta bodies for polynomial ideals, preprint,
\texttt{arXiv:0809.3480}, 2008.

\bibitem{GT89b}
M. Gr\"otschel and K. Truemper.
Master polytopes for cycles of binary matroids.
{\em Linear Algebra and its Applications}, 114/114:523--540, 1989.

\bibitem{GT89}
M. Gr\"otschel and K. Truemper.
Decomposition and optimization over cycles in binary matroids.
{\em Journal of Combinatotorial Theory B}, 46(3):306--337, 1989.


\bibitem{Las01b}
J.B. Lasserre.
An explicit exact SDP relaxation for nonlinear $0-1$ programs.
In K. Aardal and A.M.H. Gerards, eds.,
{\em Lecture Notes in Computer Science} {\bf 2081}:293--303, 2001.

\bibitem{Lau03a}
M. Laurent.
A comparison of the Sherali-Adams, Lov\'asz-Schrijver and Lasserre
relaxations for 0-1 programming.
{\em Mathematics of Operations Research} {\bf 28(3)}:470--496, 2003.

\bibitem{Lau03b} M. Laurent.  Lower bound for the number of iterations
  in semidefinite relaxations for the cut polytope.  {\em Mathematics
    of Operations Research,} {\bf 28(4)}:871--883, 2003.

\bibitem{LaurentMaxCut} M. Laurent.  \newblock Semidefinite
  relaxations for max-cut.  \newblock In {\em The sharpest cut},
  MPS/SIAM Ser. Optim., pages 257--290.  SIAM, Philadelphia, PA, 2004.

\bibitem{Lau07a}
M. Laurent.
Semidefinite representations for finite varieties.
{\em Mathematical Programming} {\bf 109}:1--26, 2007.

\bibitem{LR05} M. Laurent and F. Rendl.  Semidefinite Programming and
  Integer Programming.  In {\em Handbook on Discrete Optimization},
  K. Aardal, G. Nemhauser, R. Weismantel (eds.), pp. 393-514, Elsevier
  B.V., 2005.

\bibitem{Lo79} L. Lov\'asz.  On the Shannon capacity of a graph.  {\em
    IEEE Transactions on Information Theory} {\bf IT-25}:1--7, 1979.


\bibitem{Lovasz} L. Lov\'asz.  \newblock Semidefinite programs and
  combinatorial optimization.  \newblock In {\em Recent advances in
    algorithms and combinatorics}, volume~11 of {\em CMS Books
    Math./Ouvrages Math. SMC}, pages 137--194. Springer, New York,
  2003.

\bibitem{LS91}
L. Lov\'asz and A. Schrijver.
Cones of matrices and set-functions and
$0-1$ optimization.
{\em SIAM Journal on Optimization} {\bf 1}:166--190, 1991.

\bibitem{Oxley}
J.G. Oxley.
{\em Matroid Theory}.
Oxford University Press, Oxford, 1992.


\bibitem{Parrilo:spr}
P.A. Parrilo.
\newblock Semidefinite programming relaxations for semialgebraic problems.
\newblock {\em Math. Prog.}, 96(2, Ser. B):293--320, 2003.

%\bibitem{Sch03}
%A. Schrijver.
%Book.

\bibitem{SA90} H.D. Sherali and W.P. Adams.  A hierarchy of
  relaxations between the continuous and convex hull representations
  for zero-one programming problems.  {\em SIAM Journal on Discrete
    Mathematics} {\bf 3}:411--430, 1990.

\bibitem{Sullivant}
S. Sullivant.
\newblock Compressed polytopes and statistical disclosure limitation.
\newblock {\em Tohoku Math. J. (2)}, 58(3):433--445, 2006.

\bibitem{VB} L. Vandenberghe and S. Boyd.  Semidefinite Programming.
  {\em SIAM Review} {\bf 38(1)}:49--95, 1996.

\bibitem{W}
A. Wiegele.
{\em Nonlinear optimization techniques applied to combinatorial optimization
problems.} PhD thesis. Alpen-Adria-Universit\"at Klagenfurt, 2006.


\end{thebibliography}
\end{document}